\documentclass[final]{siamart1116}
\usepackage{latexsym,amsmath,amssymb,graphics,amscd}
\usepackage{graphicx}
\usepackage{yhmath}
\usepackage{amssymb}
\usepackage{amsfonts}
\usepackage{amsmath}
\usepackage{latexsym}
\usepackage{fancyvrb}
\usepackage{enumerate,todonotes}

\usepackage{lipsum}
\usepackage{epstopdf}
\usepackage{algorithmic}
\usepackage{verbatim}
\usepackage{mathdots}
\usepackage[normalem]{ulem}

\usepackage{xcolor}

 \newcommand{\ignore}[1]{}
\definecolor{pgreen}{rgb}{0,0.5,0}

\newcommand{\C}{\mathbb{C}}
\newcommand{\R}{\mathbb{R}}
\newcommand{\inertiA}{(n_+(A),n_0(A),n_-(A))}
\newcommand{\inertiB}{(n_+(B),n_0(B),n_-(B))}
\newcommand{\inertiAB}{(n_+(A,B),n_0(A,B),n_-(A,B),n_C(A,B),n_\infty(A,B))}
\newcommand{\inertiAaB}{(n_+(A-aB),n_0(A-aB),n_-(A-aB))}
\newcommand{\inertibBA}{(n_+(bB-A),n_0(bB-A),n_-(bB-A))}

\def\C{\mathbb{C}}
\def\R{\mathbb{R}}

\makeatletter
\newif\if@borderstar
\def\bordermatrix{\@ifnextchar*{%
 \@borderstartrue\@bordermatrix@i}{\@borderstarfalse\@bordermatrix@i*}%
}
\def\@bordermatrix@i*{\@ifnextchar[{\@bordermatrix@ii}{\@bordermatrix@ii[()]}}
\def\@bordermatrix@ii[#1]#2{%
\begingroup
 \m@th\@tempdima8.75\p@\setbox\z@\vbox{%
 \def\cr{\crcr\noalign{\kern 2\p@\global\let\cr\endline }}%
 \ialign {$##$\hfil\kern 2\p@\kern\@tempdima & \thinspace %
  \hfil $##$\hfil && \quad\hfil $##$\hfil\crcr\omit\strut %
  \hfil\crcr\noalign{\kern -\baselineskip}#2\crcr\omit %
  \strut\cr}}%
 \setbox\tw@\vbox{\unvcopy\z@\global\setbox\@ne\lastbox}%
 \setbox\tw@\hbox{\unhbox\@ne\unskip\global\setbox\@ne\lastbox}%
 \setbox\tw@\hbox{%
  $\kern\wd\@ne\kern -\@tempdima\left\@firstoftwo#1%
  \if@borderstar\kern 2pt\else\kern -\wd\@ne\fi%
 \global\setbox\@ne\vbox{\box\@ne\if@borderstar\else\kern 2\p@\fi}%
 \vcenter{\if@borderstar\else\kern -\ht\@ne\fi%
  \unvbox\z@\kern -\if@borderstar2\fi\baselineskip}%
 \if@borderstar\kern-2\@tempdima\kern2\p@\else\,\fi\right\@secondoftwo#1 $%
 }\null \;\vbox{\kern\ht\@ne\box\tw@}%
\endgroup
}

\title{
Inertia laws
and localization of 
 real eigenvalues for 
generalized indefinite 
 eigenvalue problems
}%

\author{Yuji Nakatsukasa%
\thanks{Version of \today. Mathematical Institute,
        University of Oxford,
        Oxford, OX2 6GG, UK.
        (\texttt{Yuji.Nakatsukasa@maths.ox.ac.uk})
Supported by JSPS as an Overseas Research Fellow.}
\and 
Vanni Noferini\thanks{Department of Mathematical Sciences, University of Essex, Wivenhoe Park, Colchester, CO4 3SQ, United Kingdom
                   (\texttt{vnofer@essex.ac.uk}).
                  }
    }

\date{}

\newcommand{\rank}{\text{rank}\:}

\def\C{\mathbb{C}}
\def\R{\mathbb{R}}

\makeatletter
\def\mymatrix#1{\null\,\vcenter{\normalbaselines\m@th
    \ialign{\hfil$##$\hfil&&\quad\hfil$##$\hfil\crcr
      \mathstrut\crcr\noalign{\kern-\baselineskip}
      #1\crcr\mathstrut\crcr\noalign{\kern-\baselineskip}}}\,}
\makeatother

\let\oldref\ref
\def\ref#1{{\normalfont\oldref{#1}}}
\def\eqref#1{{\normalfont(\oldref{#1})}}

\DeclareMathOperator{\sign}{sign}

\newtheorem{remark}[theorem]{Remark}

\newcommand{\TheTitle}{
Inertia laws
and localization of 
 real eigenvalues 
}
\newcommand{\TheAuthors}{Y. Nakatsukasa and V. Noferini}
\headers{\TheTitle}{\TheAuthors}

\begin{document}
\maketitle

\begin{abstract}
Sylvester's law of inertia states that the number of positive, negative and zero eigenvalues of  Hermitian matrices  is preserved under congruence transformations. 
The same is true of generalized Hermitian definite eigenvalue problems, in which the two matrices are allowed to undergo different congruence transformations,
 but not for the indefinite case. 
In this paper we investigate the possible change in inertia under congruence for generalized Hermitian indefinite eigenproblems, and derive sharp bounds that show the inertia of the two individual matrices often still provides useful information about the eigenvalues of the pencil, especially when one of the matrices is almost definite. 
A prominent application of the original Sylvester's law is in finding the number of eigenvalues in an interval. 
Our results can be used for estimating the number of real eigenvalues 
in an interval
for generalized indefinite and nonlinear eigenvalue problems. 
\end{abstract}



\section{Introduction}
Let $A=A^* \in \C^{n \times n}$. The \emph{inertia}~\cite{hornjohn} of $A$ is the triple 
\begin{equation}  \label{eq:oriinertia}
\inertiA \in \mathbb{N}^3  
\end{equation}
 where $n_+(A),n_0(A)$, and $n_-(A)$ denote, respectively, the number of positive, zero, and negative eigenvalues of $A$. Clearly, the inertia cannot be any triple of nonnegative integers, because its elements must satisfy the constraints $$n_+(A) + n_0(A) + n_-(A) = n, \qquad n_+(A) + n_-(A) = \rank(A).$$
It is also common to define the \emph{signature} of $A$ as $s(A)=n_+(A) - n_-(A)$.

Sylvester's law of inertia~\cite{hornjohn} identifies the orbits by congruence as  the equivalence classes prescribed by inertia. It is named after J. J. Sylvester, who first proved the result~\cite{syl1852}. We give below a formal statement and a concise (albeit not elementary) algebraic proof.

\begin{theorem}[Sylvester's law of inertia]\label{thm:syl}
Let $A=A^*, B=B^* \in \C^{n \times n}$. $A$ and $B$ are congruent, i.e., there exists an invertible matrix $X \in \C^{n \times n}$ such that $A=X^* B X$, if and only if $A$ and $B$ have the same inertia.
\end{theorem}

\begin{proof}
Suppose that $A$ and $B$ have the same inertia. Then, with no loss of generality we may assume that $A$ and $B$ are diagonal (by the spectral theorem) and that $A_{ii}$ and $B_{ii}$ have the same sign (if not, note that there exist a permutation matrix $P$ such that $B=P^T C P$ and $A,C$ have this property). Hence, one can define $X$ to be diagonal with $X_{ii} = (A_{ii}/B_{ii})^{1/2}$ if $B_{ii} \neq 0$ and $X_{ii}=1$ otherwise.

For the converse, observe that $n_+(A)$ is the maximal dimension of a subspace over which the quadratic form defined by $A$ is positive definite; as such, it is invariant by any change of basis (which is equivalent to a congruence on $A$). Similarly, the rank is invariant by congruence. We conclude that if $A$ and $B$ are congruent then they must have the same inertia.
\end{proof}

While the ``if'' direction of this proof is not much more than an immediate corollary of the spectral theorem, the argument to prove the ``only if'' implication is arguably more advanced, and it implicitly relies on the minmax Rayleigh characterization of the eigenvalues of a Hermitian matrix.

Sylvester's law of inertia is a useful tool in many applications, including counting the number of eigenvalues of a symmetric matrix in an interval 
and
the design of algorithms that compute eigenvalues (or singular values) with high relative accuracy~\cite[Ch.~5]{demmelbook}, counting the number of eigenvalues above (or below) a certain value. 

A few generalizations of Sylvester's law of inertia are available in the literature. Kosti\'c and Voss~\cite{kosti2013sylvester} introduced Sylvester-like laws of inertia for nonlinear eigenvalue problems, but they dealt with situations in which a minmax-type characterization of the eigenvalues exist. Instead, we are mainly concerned with those 
that do not fall into this category, for example the generalized indefinite eigenvalue problem,  so that the results in \cite{kosti2013sylvester} are inapplicable. 
Bilir and Chicone~\cite{bilirchicone98} studied the inertia of the real parts of the eigenvalues of quadratic matrix polynomials whose leading and trailing coefficients are Hermitian, and the second coefficient has a positive definite Hermitian part.
Going back to matrix, as opposed to nonlinear, eigenvalue problems, Ikramov~\cite{ikramov01} showed that
two normal matrices are congruent if and only if they have the same number of eigenvalues on any semiline starting from $0$ in the complex plane. 

In this work, we aim to extend Sylvester's classical result in a different direction. Specifically, we will analyze Hermitian (\emph{indefinite}) generalized, polynomial, and generally nonlinear eigenvalue problems. We will see that, although it is not generally possible to determine the number of positive, zero, negative, and overall real eigenvalues, there exist nontrivial bounds on these quantities that only depend on the inertia of the trailing and (when appropriate) leading coefficients of the eigenvalue problem, and as such are invariant by congruence on them. 


As in the original law of inertia, our approach can be used to obtain bounds for the number of eigenvalues lying in an interval for generalized Hermitian indefinite eigenvalue problems, which are particularly useful when one of the matrices is nearly positive (or negative) definite. Moreover, we show that such results can be extended to nonlinear eigenvalue problems. 

Our results are useful in a number of applications. For example, estimating the number of eigenvalues in an interval is a key component for an efficient eigensolver based on splitting the spectrum, for example using contour integration; see~\cite{di2016efficient} and the references therein. 
Generalized Hermitian indefinite eigenvalue problems arise for instance 
in optimization problems~\cite{adachi2017solving}, 
tensor decomposition~\cite{Lathauwer2006}, and
when one uses a structured-preserving linearization of a Hermitian matrix polynomial (e.g., one arising from a bivariate zerofinding problem~\cite{biroots}), for example the $\mathbb{DL}$ linearization~\cite{Mackey05vectorspaces, m4revsimax} or the family of block symmetric linearizations described in~\cite{dopicosymlin}. In all these applications \cite{adachi2017solving}, \cite{Lathauwer2006} and \cite{biroots}, it is the real eigenvalues that are of interest. 
Eigenvalues of indefinite pairs play a crucial role also in preconditioning, for example when a constraint preconditioner is used for indefinite linear systems~\cite{keller2000constraint}. 


The structure of the paper is as follows. In Section~\ref{sec:classical} we revisit Sylvester's law of inertia for classical and generalized Hermitian definite eigenvalue problems. In particular, we give an analytic proof  of the ``only if'' implication in Theorem~\ref{thm:syl} by using elementary arguments, namely continuity of eigenvalues and the intermediate value theorem. In Section~\ref{sec:generalized} we extend the argument to analyze the generalized indefinite eigenvalue problem and obtain bounds for its inertia. 
In Section~\ref{sec:applications} we discuss applications and variants, including bounding the number of eigenvalues 
in any real interval, i.e., not necessarily $(0,\infty)$ or $(-\infty,0)$.
In Section~\ref{sec:nonlinear} we discuss the polynomial eigenvalue problem, and more generally the nonlinear eigenvalue problem. 
Until Section~\ref{sec:nonlinear} we focus on counting eigenvalues with their algebraic multiplicity; Section~\ref{sec:geom} discusses bounds when counting eigenvalues using their geometric multiplicity. Numerical experiments are presented in Section~\ref{sec:exp} to illustrate the results.

\section{Sylvester's law of inertia revisited}\label{sec:classical}

The following result (see~\cite[Sec. II.5]{kato} and~\cite[Sec. 5]{MNTX16} for a more thorough discussion) is key for our analysis.

\begin{theorem}\label{thm:conteigs}
Let $\Omega \subseteq \R$ be an interval, and let $F(t)$ be an $n \times n$ complex matrix whose entries depend continuously on a real parameter $t$ and such that for all $t \in \Omega$ the eigenvalues of $F(t)$ are real. Then, there exist $n$ continuous functions $\lambda_i(t)$, $i=1,\dots,n$, such that for all $t \in \Omega$ they are the eigenvalues of $F(t)$.
\end{theorem}

In this paper, we will take $F(t)$ to be Hermitian for all values of $t \in \Omega$, which of course guarantees that its eigenvalues are real. 
F. Rellich~\cite{rellich37,rellichbook} pioneered the study of this special case, and we will sometimes call the $\lambda_i(t)$ ``Rellich's eigenfunctions'', or just ``eigenfunctions'' for brevity\footnote{In other contexts, ``eigenfunction'' may mean the eigenvector of linear operators defined on functional spaces. We never use the term in that sense in this paper, so no ambiguity should arise.}.

We first set $F(t) = A - I t$ where $A=A^* \in \C^{n \times n}$ has inertia $\inertiA$. Clearly, the assumptions of Theorem~\ref{thm:conteigs} are satisfied on $\Omega = \R$, and it is particularly easy to write down explicitly the functions 
$$\lambda_i(t) = \lambda_i(A) - t,$$
where $\lambda_i(A)$ is the $i$th eigenvalue of $A$. With this view, the eigenvalues of $A$ are precisely the values of $t$ for which one of Rellich's eigenfunctions $\lambda_i(t)$ of $F(t)$ has a zero. See Figure~\ref{fig:sym} (left) for an illustration.

\begin{figure}
  \begin{minipage}[t]{0.5\hsize}
      \includegraphics[width=0.9\textwidth]{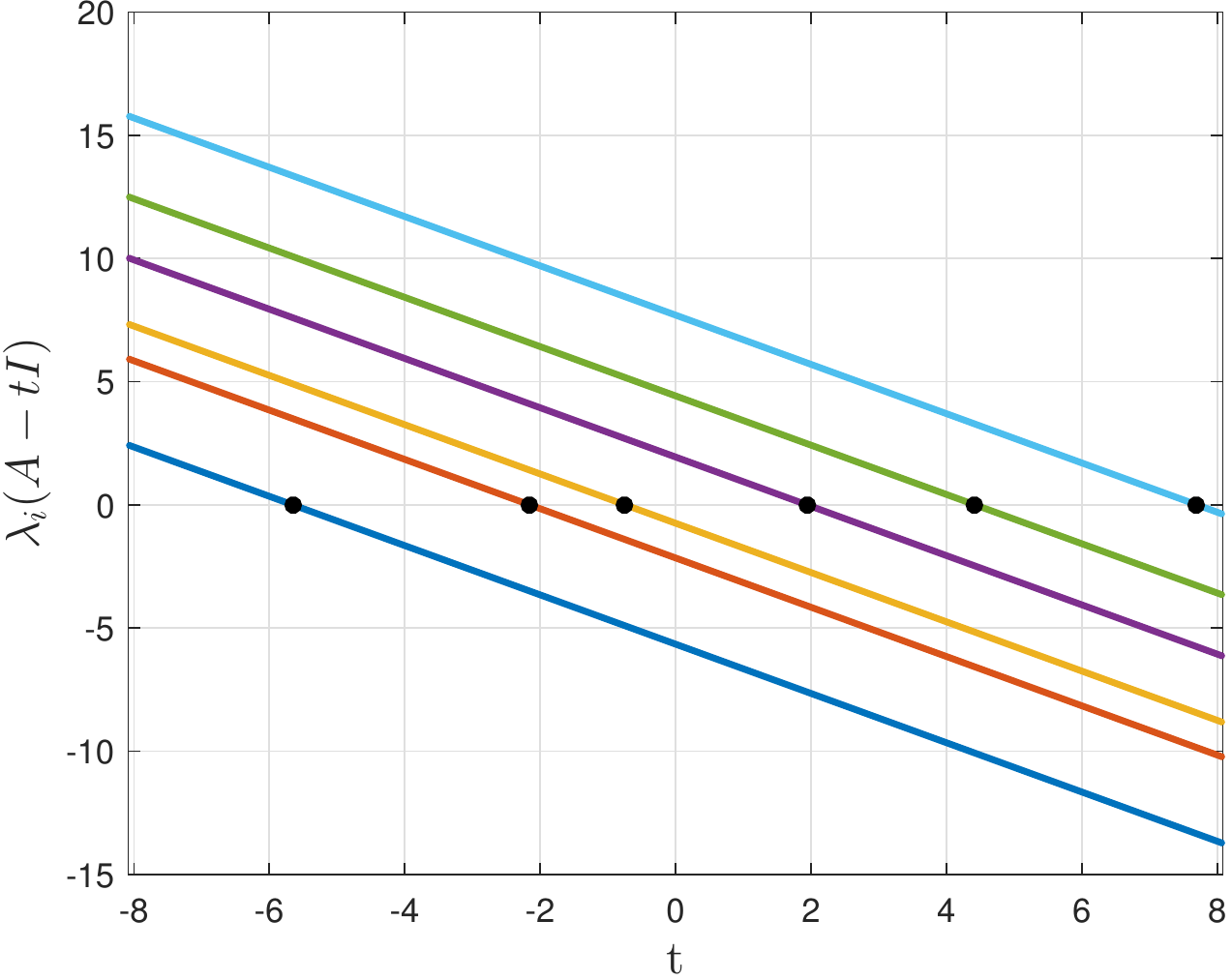}
  \end{minipage}   
  \begin{minipage}[t]{0.5\hsize}
      \includegraphics[width=0.9\textwidth]{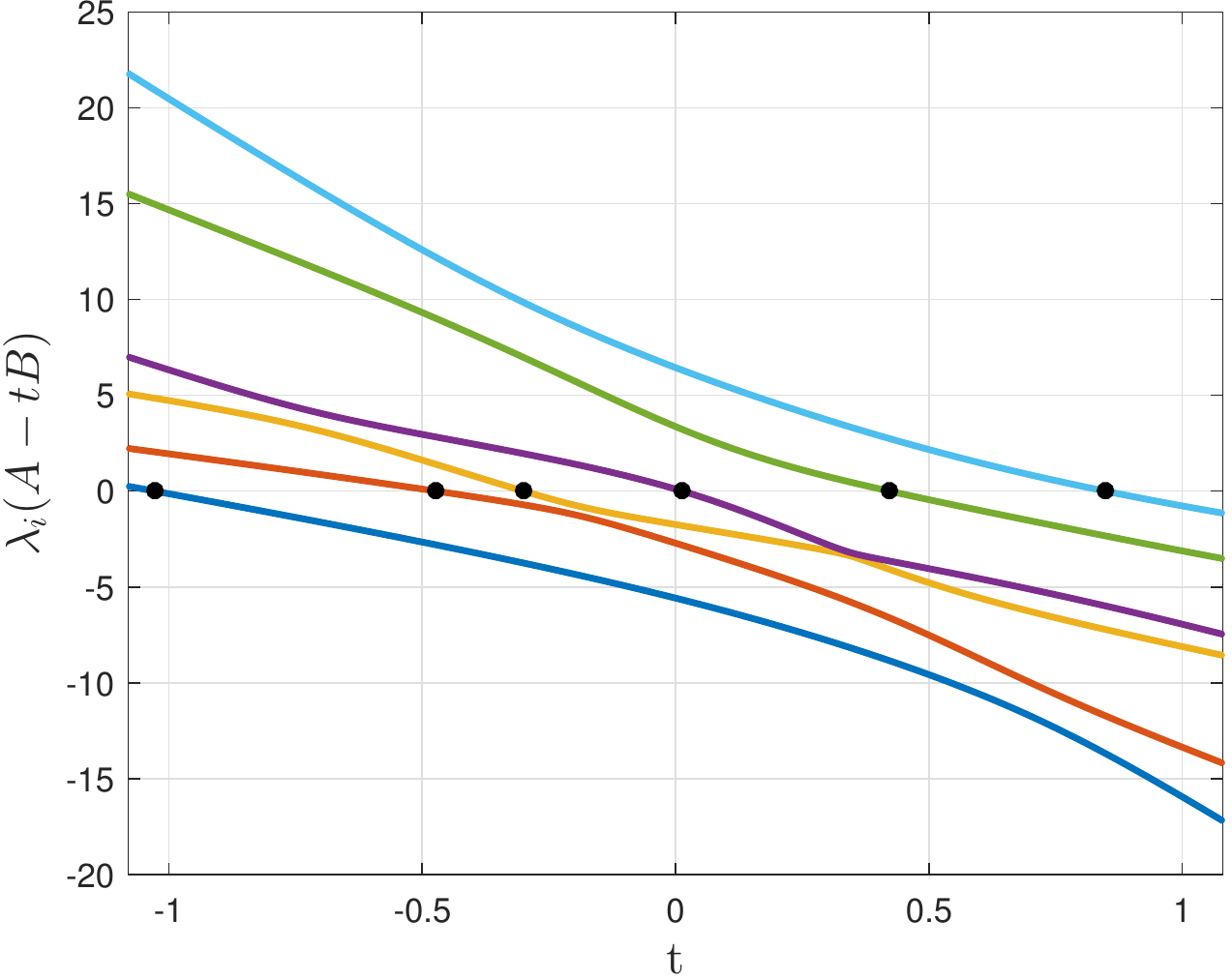}
  \end{minipage}
  \caption{Eigenvalues of $6\times 6$ matrices of form $A-tI$ (left) and $A-tB$, $B\succ 0$ (right). The eigenvalues of $(A,B)$, shown as the abscissae of the black dots, are the roots of the curves. The inertiae of $A$, and hence of the pencil, 
are $(3,0,3)$ in both figures.}
  \label{fig:sym}
\end{figure}


Our second step is to consider a generalized Hermitian definite eigenvalue problem, and to study its eigenvalues (whose formal definition is given in Section \ref{sec:generalized}). This corresponds to $F(t)=A-t B$, with $B=B^* \neq I$ but $B \succ 0$. The eigenfunctions $\lambda_i(t)$ are no longer necessarily affine functions, but they are still strictly decreasing functions. This can be proved, for instance, by Weyl's theorem. See Figure~\ref{fig:sym} (right).
In particular, for this choice of $F(t)$ one has that $\lambda_i(t)\rightarrow-\infty$ as $t\rightarrow +\infty$ and $\lambda_i(t)\rightarrow+\infty$ as $t\rightarrow -\infty$; implying that if a curve $\lambda_i(t)$ takes a positive value (resp. negative) at $t=0$, then it must have a root on $(0,\infty)$ (resp. $(-\infty,0)$). 
Therefore the number of positive (resp., zero, negative) eigenvalues is still exactly equal to that of $A$. This observation constitutes an analytic proof of the ``only if'' part of Theorem~\ref{thm:syl}. Indeed, it suffices to observe that $\det (X^* A X - t I) = \det(X)^2 \det(A - t X^{-*} X^{-1})$ and that $X^{-*}X^{-1}$ is positive definite by construction.


So far, we have provided an unusual analytic approach to the subject of inertia to reinvent the wheel, but we have not obtained any new results. Crucially, though, this technique relies on the very general Theorem~\ref{thm:conteigs}, and therefore it is particularly suitable to be carried over to other types of eigenvalue problems. Before proceeding with this task, let us gain further insight into the difficulties that arise in locating the eigenvalues of the pencil when the definiteness assumption is dropped by setting $F(t)=A-tB$ with $A$ and $B$ Hermitian, but $B$ no longer positive definite\footnote{Note that the cases of  $B$ negative definite or $A$ (positive or negative) definite are also trivial, in the sense of being analogous to the the case $B \succ 0$. For example, we can set, respectively, $F(t) = A + t B$, $F(t) = tA  - B$, $F(t) = tA  + B$.
Similarly, if a linear combination $\alpha A+\beta B$ for some $\alpha,\beta\in\mathbb{R}$ is positive definite, then one can consider $F(t)=A-t(\alpha A + \beta B)$: indeed, note that the eigenfunctions of the latter have a zero at $\frac{\lambda}{\alpha+\beta \lambda}$ if and only if $\lambda$ is an eigenvalue of the generalized eigenvalue problem $A-z B$.
}. Proofs analogous to the one we gave for Theorem~\ref{thm:syl}, relying on the minmax characterization, will clearly face the difficulty that such characterizations are unavailable.

To illustrate the idea, suppose that $B$ has only $1$ negative eigenvalue and $n-1$ positive eigenvalues. Then, our results show that at least $n-2$ eigenvalues are real; see the left plot of Figure~\ref{fig:symindef} for an illustration with $n=6$. The inertia of $A$ is $(3,0,3)$, so three of the six curves $\lambda_i(A-tB)$ take positive values at $t=0$. As $t\rightarrow\infty$, however, the inertia of $A-tB$ must match that of $-B$, which is $(1,0,5)$. Hence, five curves must be negative for $t$ sufficiently large.  It follows from the intermediate value theorem that at least two curves must intersect the $t$-axis. We conclude that at least two of the generalized eigenvalues of the pencil $(A,B)$ are positive, and analogously the pencil has at least two negative eigenvalues. 
In this example the bound matches the exact value, but to be convinced that such lower bound may be an underestimate, see the right plot of Figure~\ref{fig:symindef}. We have kept the same $A$, but changed $B$ so that it now has inertia $(3,0,3)$. The intermediate value theorem gives the trivial lower bound zero for the number of positive eigenvalues. Nonetheless, $(A,B)$ has two positive eigenvalues.

\begin{figure}
  \begin{minipage}[t]{0.5\hsize}
      \includegraphics[height=50mm]{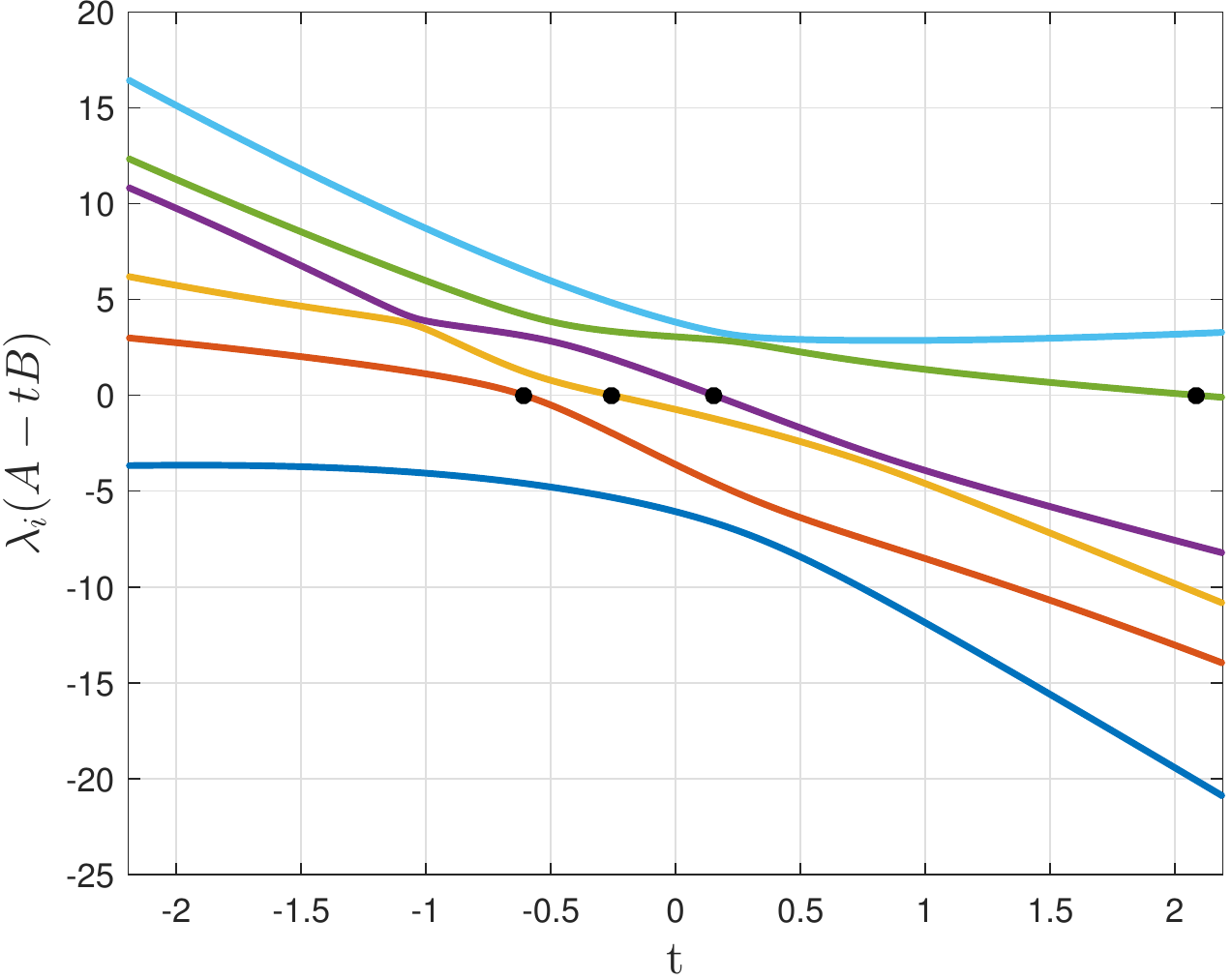}
  \end{minipage}   
  \begin{minipage}[t]{0.5\hsize}
      \includegraphics[height=50mm]{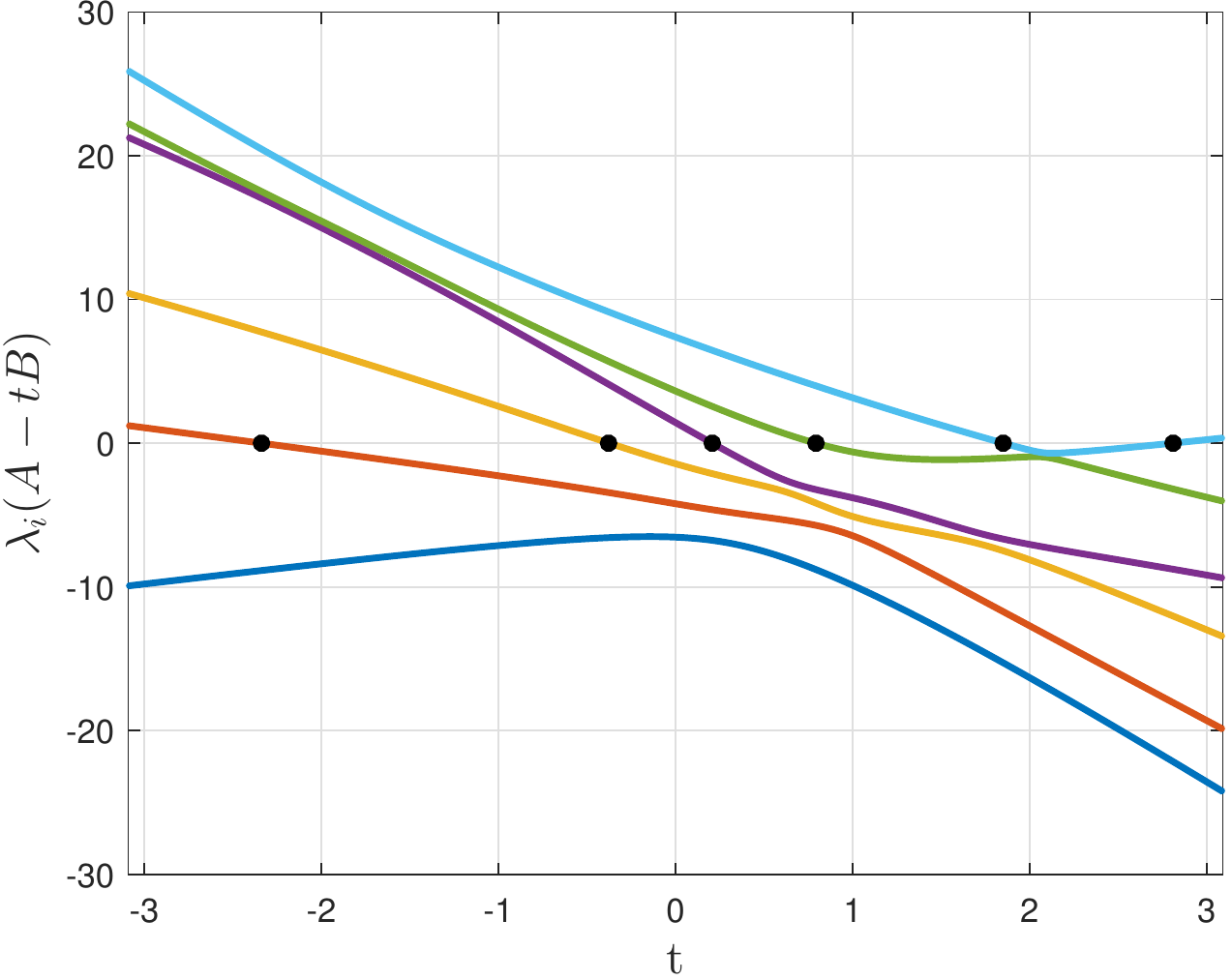}
  \end{minipage}
  \caption{Plot of the eigenvalues of $A-tB$ for $A,B$ both indefinite.
 The inertia of $A$ is $(3,0,3)$ in both cases. 
Left: $B$ has inertia $(5,0,1)$. Our result proves that $(A,B)$ must have at least two negative eigenvalues, and the bound is attained in both these examples. 
Right: $A$ and $B$ have inertia $(3,0,3)$. Our result gives a lower bound of zero negative eigenvalues, but there are two of them. 
}
  \label{fig:symindef}
\end{figure}

In Section~\ref{sec:generalized}, we make these observations more general and more precise.

\section{Inertia bounds for indefinite pencils}\label{sec:generalized}

Given $A=A^*, B=B^* \in \C^{n \times n}$, the Hermitian generalized eigenvalue problem is to find the values of $z$ such that the rank of $A-z B$ is not maximal. Formally, we define $A-\infty B := B$ and $\C^*:=\C \cup \{\infty\}$. We say that $\lambda \in \C^*$ is an eigenvalue of the pencil $A-zB$, or equivalently of $(A,B)$, if
\begin{equation}  \label{eq:defeig}
\rank(A-\lambda B)<\max_{z \in \C^*} \rank (A-z B) =: r.  
\end{equation}
This generalizes the standard definition $Ax=\lambda Bx$, $0 \neq x \in \C^n$, which is valid for \emph{regular} generalized eigenvalue problems, i.e., when $r=n$.
The right-hand side of the above equation, $r$, is sometimes called the \emph{normal rank}, or simply rank, of the pencil $A-z B$.
Note that, invoking Theorem~\ref{thm:conteigs}, this definition implies that a real eigenvalue must be a  zero of at least one non-identically zero eigenfunction $\lambda_i(t)$ of $F(t)=A-tB$, $t \in \R$. The definition~\eqref{eq:defeig} also ensures that there are at most $r$ eigenvalues, counted with (either algebraic or geometric) multiplicities~\cite{dopicoindexsum}. 
Here, the geometric multiplicity of an eigenvalue $\lambda$ is the rank drop 
$r - \rank(A-\lambda B)$. The algebraic multiplicity is the sum of the multiplicities of the zeros of the non-identically zero eigenfunctions; see~\cite[Def.~2.3]{MNTX16}. Clearly, the algebraic multiplicity is at least as large as the geometric counterpart.

There may be strictly less than $r$ eigenvalues; for example, if 
\begin{equation}  \label{eq:abzsingular}
A-zB  = \begin{bmatrix}
0 & 1 & -z\\
1 & 0 & 0\\
-z & 0 & 0
\end{bmatrix}  
\end{equation}
 there are no eigenvalues, even though the rank is $2$. 
See, e.g., \cite{dopicoindexsum} for a sharp characterization of the number of the eigenvalues.

As we have seen in Section~\ref{sec:classical}, when $B$ (or $A$) is positive (or negative) definite, the number of positive, zero, and negative eigenvalues can be determined from the inertia of $A$ (or $B$), and this task has a close relationship with Sylvester's classical result on the inertia of matrices. Similar remarks can be made when there exist scalars $\alpha,\beta \in \mathbb{R}$ such that $\alpha A+\beta B$ is positive definite; we will discuss this situation in detail in Section~\ref{sec:applications}.

When no such $(\alpha,\beta)$ can be found,  the generalized eigenvalue problem is said to be indefinite, and it becomes somewhat more difficult. 
One complication with generalized indefinite eigenvalue problems is the potential existence of nonreal eigenvalues and eigenvalues at infinity. Another is that the pencil may be singular, which means that the normal rank is deficient, or equivalently, that $\det (A-zB) = 0$ for all $z\in\mathbb{C}$.  For example, the pencil~\eqref{eq:abzsingular} is singular.
\begin{remark}
Having a connection with Sylvester's law of inertia in mind, we also note that it may happen that $A- z B$ is a regular pencil, i.e. $\det (A- z B)$ is not identically zero, but there exist invertible $X,Y$ for which 
the pencil $X^*AX-z Y^*B Y$ is singular. For example, take
\begin{equation}\label{eq:mortimer}
A=\begin{bmatrix}
1 & 0\\
0 & 0
\end{bmatrix},\quad B=\begin{bmatrix}
0 & 0\\
0 & 1
\end{bmatrix},\quad X=I_2,\quad Y=\begin{bmatrix}
0 & 1\\
1 & 0
\end{bmatrix}.
\end{equation} 
The results that we develop in the current section provide localization bounds for the eigenvalues that only depend on the inertia of the two matrix coefficients of the generalized eigenvalue problem, and therefore are invariant under such congruence transformations of the pencil. In other words, we give bounds that
 hold regardless of whether or not the resulting pencil $X^*AX - z Y^*B Y$ is singular.
\end{remark}

In order to take into account eigenvalues that are nonreal and at infinity, we define the inertia of $A-z B$ as the quintuple
\begin{equation}  \label{eq:inertiadef}
  \inertiAB. 
\end{equation}
The entries of the inertia of $A-z B$ are nonnegative integers that count the number of, resp., positive, zero, negative, complex, and infinite eigenvalues, where each eigenvalue is counted with its algebraic multiplicity. 
Most, but not all, bounds we derive are also valid for the geometric multiplicity; we describe this in Section~\ref{sec:geom}.
We note that our definition is different from the one by Bilir and Chicone~\cite{bilirchicone98}, who examined the real parts of the eigenvalues (and hence no $n_C(A,B)$ arises). 

There appear to be no known 
 bounds for the inertia \eqref{eq:inertiadef} of generalized indefinite eigenvalue problems. 
The purpose of Theorem~\ref{thm:main} below is to fill in this gap.

\begin{theorem}\label{thm:main}
Let $A,B$ be Hermitian matrices and denote their inertia by $\inertiA$ and $\inertiB$. Also define for notational convenience

\[ n_{++} = n_+(A) + n_+(B), \quad 
n_{--}=n_-(A)+n_-(B), \]  
\[ n_{+-} = n_+(A) + n_-(B),\quad 
n_{-+}= n_-(A)+n_+(B),   \]
\[ \delta = n_0(A) - n_0(B), \]
and $N_{++}=\max(n_{++},n_{--}), N_{+-}=\max(n_{+-},n_{-+})$.
 Then the inertia of the generalized eigenvalue problem $A - z B$ as in \eqref{eq:inertiadef} satisfies 
  \begin{align}
N_{++} -n 
&\leq n_+(A,B) \leq 2 n - |\delta| - N_{+-},   \label{eq:posin}     \\
\delta& \leq  n_0(A,B)\leq 3n - N_{++} - N_{+-} - n_0(B),   \label{eq:zeroin}    \\
N_{+-}- n &\leq n_-(A,B) \leq 2 n - |\delta| - N_{++},\label{eq:negin}    \\
0 &\leq  n_C(A,B)\leq 2 \min(n_+(A),n_-(A),n_+(B),n_-(B)),   \label{eq:compin}    \\
-\delta &\leq  n_\infty(A,B)\leq 3n - N_{++}- N_{+-} - n_0(A).   \label{eq:infin}    
  \end{align}
  Moreover, denoting by
 $n_R(A,B)$ the number of finite real eigenvalues of the generalized eigenvalue problem $A-z B$, we have 
\begin{equation}\label{eq:realin}
|s(B)|
\leq  n_R(A,B)\leq n - n_0(B).  
\end{equation}
\end{theorem}

\begin{remark}
Of course, we also have the trivial bounds: lower bound 0 and upper bound $n$. 
It is possible that some of the bounds~\eqref{eq:posin}--\eqref{eq:compin} are no more informative. 
\end{remark}

\begin{proof}[Proof of Theorem~\ref{thm:main}]
We start from zero and infinite eigenvalues. We need the following definitions: let $r$ be the rank of the pencil $A - z B$. Let $N=n_+(A,B) + n_-(A,B) + n_0 (A,B) + n_\infty (A,B) + n_C(A,B)$ be the total number of eigenvalues counted with algebraic multiplicities. Then $N \leq r$. The number of intersections of the Reillich's eigenfunctions with the real axis is equal to, 
counting multiplicities,
$N-n_C(A,B) - n_\infty(A,B)$. Futhermore, there are $n-r$ eigenfunctions identically equal to zero.

The number of Rellich's eigenfunctions which are $0$ at $t=0$ is $n_0(A)$; of those, $n-r$ correspond to identically zero eigenfunctions. Hence, it is immediate that $n_0(A,B) \geq n_0(A) - (n-r)$.
Using $\rank(B) = n - n_0 (B) \leq r$, we note that $n-r \leq n_0(B)$, and hence,  
\[ n_0(A,B) \geq n_0(A) - n_0(B).\]
Switching the roles of $A$ and $B$ yields the lower bound for $n_\infty(A,B)$.

Now we come to the positive eigenvalues. To avoid normalization issues at infinity, it will be convenient to consider the mapping $A - z B \mapsto \sin(\theta) (A - \cot(\theta) B)$ and look at the eigenvalues of $A \sin (\theta) - B \cos (\theta)$ for $\theta \in [0,\pi]$, and at the corresponding eigenfunctions as functions of $\theta$. The number of positive eigenvalues is the number of zeros in $(0,\pi/2)$ of the non-identically zero eigenfunctions. Let us classify each eigenfunction according to its sign at $\theta=\pi/2$ ($z=0$) and its sign at $\theta=0$ ($z=+\infty$). Note that there are 9 possible types.

Any eigenfunction that is positive at $\theta=\pi/2$ corresponds to a positive eigenvalue of $A$; in order for this eigenfunction not to correspond to a real positive eigenvalue of the pencil, then it must be $\geq 0$ at $\theta=0$, corresponding to a nonpositive eigenvalue of $B$. Hence, we have the lower bound

\[  n_+(A,B) \geq n_+(A) - n_0(B) - n_-(B) = n_+(A) + n_+(B) - n.  \] 

Since an analogous argument can be given by switching the roles of $+$ and $-$, we also obtain

\[  n_+(A,B) \geq  n_-(B) + n_-(A) - n.  \]

The lower bounds for the negative eigenvalues are obtained immediately via the mapping $B \mapsto -B$.

For the upper bounds, observe that the number of positive eigenvalues is bounded above by the total number of eigenvalues, minus the number of infinite, zero and negative eigenvalues. Using the bounds previously obtained this yields

\[ n_+(A,B) \leq r - |\delta| + \min(n-n_-(A)-n_+(B),n-n_+(A)-n_-(B))  \]

Observe now that $r \leq n$ to obtain the bound for positive eigenvalues. The bound for negative eigenvalues is obtained similarly by sending $B \mapsto -B$. The upper  bounds for zero and infinite eigenvalues are obtained analogously, by using the lower bounds. We omit the details.

We turn to \eqref{eq:realin}.
For the upper bound, 
using the lower bound for $n_\infty(A,B)$ in~\eqref{eq:infin} with $\rank(A) = n - n_0 (A)\leq r$ we obtain
$$n_R(A,B) \leq r - n_\infty(A,B) \leq r - n_0(B)+n_0(A) \leq n - n_0(B).$$ 
For the lower bound $s(B)$, observe that, setting $F(t)=A-tB$ in Theorem~\ref{thm:conteigs}, the Rellich's eigenfunctions $\lambda_i(t)$ must (altogether) have at least $|n_+(B)-n_-(B)|$ zeros on $(-\infty,\infty)$. 

It remains to establish the upper bound for the number of complex eigenvalues~\eqref{eq:compin}. 
We trivially have $n_C(A,B)\leq n-n_R(A,B)-n_\infty(A,B)$. 
Now since $n_\infty(A,B)\geq n_0(B)-(n-r)$, 
 using~\eqref{eq:realin} we obtain
\[
n_C(A,B)\leq r-n_R(A,B)-n_\infty(A,B) \leq n-n_0(B)-|n_+(B)-n_-(B)|. 
\]
Consider the case $n_+(B)>n_-(B)$. Then since $n_+(B)=n-n_-(B)-n_0(B)$, we obtain 
\[
n_C(A,B)\leq n-n_0(B)-(n_+(B)-n_-(B)) 
= 2n_-(B). 
\]
The case $n_+(B)>n_-(B)$ is analogous. 
By symmetry of the statement~\eqref{eq:compin}
with respect to the matrices $A,B$, we also conclude that 
$n_C(A,B)$ is bounded by $2n_-(A)$ and  $2n_+(A)$. 
\end{proof}

It is worth remarking that all the bounds that appear in Theorem~\ref{thm:main} 
are invariant under congruence. Therefore all the bounds remain the same when applied to $(A,B)$ or $(X^*AX,Y^*AY)$ for nonsingular matrices $X,Y$. 
In particular, if $B \succ 0$ then $B$ is congruent to $I$ and we recover the classical Sylvester's theorem of inertia: the bounds~\eqref{eq:posin}--~\eqref{eq:negin} become equalities on $n_-(X^*AX,Y^*BY),n_0(X^*AX,Y^*BY)$ and $n_+(X^*AX,Y^*BY)$. Otherwise (when $B$ is not positive definite), the congruence transformation may result in different $n_-(X^*AX,Y^*BY),n_0(X^*AX,Y^*BY)$ and $n_+(X^*AX,Y^*BY)$, but within the intervals given in~\eqref{eq:posin}--~\eqref{eq:negin}. 
 
 \begin{remark}\label{canonical}
 An alternative, and more algebraic, approach to proving Theorem~\ref{thm:main} is possible via the canonical form of Hermitian pencils via congruence~\cite{lancaster2005canonical}. We prefer the analytic approach, because it generalizes to polynomial and nonlinear eigenvalue problems, for which no canonical form is available.
 \end{remark}

\begin{theorem}\label{thm:attainable}
The lower bounds in \eqref{eq:posin}--\eqref{eq:infin} in Theorem~\ref{thm:main} are sharp: given any Hermitian pair $A,B$ and given any of the corresponding lower bounds in \eqref{eq:posin}--\eqref{eq:infin}, suppose that the bound is not trivial (i.e., it is not below $0$ or above $r$). Then, there exist nonsingular matrices $X,Y$ such that
for the pencil $X^* A X - z Y^* B Y$, 
 equality is attained for the bound considered. In particular, it is impossible to obtain a nontrivial lower bound for 
$n_C(A,B)$ from the inertia information of $A,B$. 
\end{theorem}

\begin{proof} The proof is constructive: for each lower bound, we exhibit a Hermitian pair $(\hat{A},\hat{B})$ such that $A$, resp. $B$, has the same inertia as $\hat{A}$, resp. $\hat{B}$, and the pencil $\hat{A}-z\hat{B}$ achieves the bound. The result will then follow by Theorem~\ref{thm:syl}.

\begin{itemize}
\item \emph{Equations \eqref{eq:posin}, \eqref{eq:negin}, \eqref{eq:compin}.}
Let us first show how the lower bound on positive eigenvalues can be attained. Without loss of generality, we take 
$n_{++}\geq n_{--}$
 and $n_+(A) \geq n_-(B)$. Indeed, note that 
$A-zB$, $zB-A$, $B-zA$, and $zA-B$
 all have the same number of positive eigenvalues, so if necessary we can swap the roles of $A$ and $B$ and/or multiply both matrices times $-1$.
We take
$$ \hat{A} = \begin{bmatrix}
I_{n_+(A)} & 0 & 0\\
0 & -I_{n_-(A)} & 0\\
0 & 0 & 0
\end{bmatrix}, \qquad \hat{B} = \begin{bmatrix}
-I_{n_-(B)} & 0 & 0\\
0 & 0 & 0\\
0 & 0 & I_{n_+(B))}
\end{bmatrix}. $$
By the assumptions above, the only possibility for a positive eigenvalue is if the blocks $I_{n_+(A)}$ and $I_{n_+(B)}$ superpose, which either happens precisely on a sub-block of size 
$n_{++}-n$
 (if 
this is nonnegative), 
 or it does not happen at all otherwise. Note that this example also attains the lower bound of $0$ complex eigenvalues, and that again 
$A-zB$, $zB-A$, $B-zA$, and $zA-B$ all have the same number of complex eigenvalues so we are not losing in generality.

For the lower bound on negative eigenvalues, we assume without loss of generality that 
$n_{+-}\geq n_{-+}$ and $n_+(A) \geq n_+(B)$, and we change the matrices from the example above as follows: $(\hat{A},\hat{B}) \rightarrow (\hat{A},-\hat{B})$.
\item \emph{Equations \eqref{eq:zeroin}, \eqref{eq:infin}.} We only discuss the lower bound for zero eigenvalues, as the case of infinite eigenvalues can be proved similarly switching the roles of $A$ and $B$. We take $$ \hat{A} = \begin{bmatrix}
0 & 0 & 0\\
0 & I_{n_+(A)} & 0 \\
0 & 0 & -I_{n_-(A)}\\
\end{bmatrix}, \qquad \hat{B} = \begin{bmatrix}
0 & 0 & 0\\
0 & I_{n_+(B)} & 0\\
0 & 0 & -I_{n_-(B))}
\end{bmatrix}, $$
and we observe that $\hat{A} - z \hat{B}$ has no zero eigenvalues if $n_0(A) \leq n_0(B)$, and it has $\delta$ zero eigenvalues otherwise.

\end{itemize}

\end{proof}

\begin{remark}
For the lower bound $|s(B)|\leq  n_R(A,B)$ in~\eqref{eq:realin}, the bound is sharp if we assume the knowledge of only $\inertiB$, i.e, no inertia information is available on $A$. For instance, take 
\[
\hat{A} = 
\begin{bmatrix}
R & 0 & 0\\
0 & D_1 & 0\\
0 & 0 & D_2
\end{bmatrix}
,\qquad
\hat{B} = 
\begin{bmatrix}
T & 0 & 0\\
0 & \pm I_{|s(B)|} & 0\\
0 & 0 & 0 
\end{bmatrix}
\]
where $p=\min(n_+(B),n_-(B))$, the sign in $\pm I_{|s(B)|}$ takes $+1$ if $n_+(B)> n_-(B)$ and $-1$ otherwise (the block is empty if $s(B)=0$), 
$D_1$ and $D_2$ are real diagonal matrices whose precise entries do not matter, 
and 
\begin{equation}
  \label{eq:RT}
R  = 
\begin{bmatrix}
0 & 1\\
1 & 0
\end{bmatrix} \oplus \dots \oplus \begin{bmatrix}
0 & 1\\
1 & 0
\end{bmatrix}  \in \C^{2p \times 2p},\quad 
T =
\begin{bmatrix}
1 & 0\\
0 & -1
\end{bmatrix} \oplus \dots \oplus \begin{bmatrix}
1 & 0\\
0 & -1
\end{bmatrix}   \in \C^{2p \times 2p}.      
\end{equation}

Obtaining a sharp bound using the inertia of both $A$ and $B$ appears to be more difficult, as the possible presence of eigenvalues at $0$ and $\infty$ complicates the situation. We clearly have the lower bound $n_R(A,B)\geq N_{++}+N_{+-} -2n+\delta$ obtained by adding~\eqref{eq:posin}--\eqref{eq:negin}, but neither this nor $|s(B)|$ is always sharp.

If we assume further that $0$ and $\infty$ are not eigenvalues of $(A,B)$, then we obtain the improved and sharp lower bound 
\begin{equation}
  \label{eq:sharplo}
n_R(A,B)\geq |n_+(B)-n_-(A)|+|n_-(A)-n_-(B)|.   
\end{equation}
 The bound is a straightforward consequence of the intermediate value theorem, counting the number of eigenfunctions taking negative values at $t=\infty,0,-\infty$, and noting (by assumption) that no nonzero eigenfunction is 0 at these three points. 
We have $|s(B)|=|n_+(B)-n_-(B)|\leq |n_+(B)-n_-(A)|+|n_-(A)-n_-(B)|$ by the triangular inequality, and the sharpness of the bound~\eqref{eq:sharplo} can be verified by taking
$$ \hat{A} = \begin{bmatrix}
R & 0 & 0& 0\\
0 & D_1 & 0& 0\\
0 & 0 & D_2& 0 \\
0 & 0 & 0& 0
\end{bmatrix}, \qquad \hat{B} = \begin{bmatrix}
T & 0 & 0& 0\\
0 & I_{n_+(B)-p} & 0& 0\\
0 & 0 & -I_{n_-(B)-p}& 0 \\
0 & 0 & 0 & 0
\end{bmatrix}, $$ where 
$R,T$ are as above with $p=\min(n_+(A),n_-(A),n_+(B),n_-(B))$. 
$\hat{A}-z\hat{B}$ has $n_+(B)+n_-(B)-2p=r-2p$ real eigenvalues, and a direct calculation shows that $r-2p=|n_+(B)-n_-(A)|+|n_-(A)-n_-(B)|$, as required.


If assuming that $0,\infty$ are not eigenvalues is undesirable, one can work with a shifted pencil:
one analogously obtains
$n_R(A,B)\geq |n_+(B-\epsilon A)-n_-(A-\epsilon B)|+|n_-(A-\epsilon B)-n_-(B-\epsilon A)|$ for an arbitrary $\epsilon>0$, which is sharp assuming $\epsilon,1/\epsilon$ are not eigenvalues and $(-1/\epsilon,1/\epsilon)$ contains all finite real eigenvalues of $(A,B)$.
\end{remark}

\begin{remark}
Some of the upper bounds in Theorem~\ref{thm:main} are also sharp in the sense of Theorem~\ref{thm:attainable}. Specifically:
\begin{itemize}
\item The upper bounds in \eqref{eq:posin} and \eqref{eq:negin} are not always attainable. For example, if $n=2m$ is even, $\inertiA=(m,m,0)$ and $\inertiB=(0,m,m)$, our upper bound for $n_+(A,B)$ is $n$, but it is possible to prove with other means that no more than $m$ positive eigenvalues can possibly be present.
\item The upper bounds in \eqref{eq:zeroin} and \eqref{eq:infin} are not always attainable. For example, suppose that $n=2m$ is even and $\inertiA=\inertiB=(m,0,m)$. Then, neither zero nor infinite eigenvalues can be present, but our upper bounds are equal to $n$.
\item The upper bound in \eqref{eq:compin} is attained by the very same example as in the lower bound of \eqref{eq:realin}.
\item The upper bound in \eqref{eq:realin} is attained, for example, by any choice of $\hat{A}$ and $\hat{B}$ with both matrices diagonal and with $1,-1,0$ elements on the main diagonal.
\end{itemize}

We also note that, although we opted in Theorem~\ref{thm:main} not to give bounds that depend on the normal rank $r$, the latter 
 can sometimes be computed (for example as  $r = \max_i \rank (A-\mu_i B)$ where $\mu_0,\dots,\mu_n$ are any $n+1$ distinct complex numbers; in practice, 
with one randomly chosen $\mu_0\in\mathbb{C}$, we have $r = \rank (A-\mu_0 B)$  with high probability).
 In this scenario, the following bounds, obtained from the intermediate steps in the proof of Theorem~\ref{thm:main}, are potentially more informative:
 \begin{align}
n_+(A,B) &\leq 3n -r -n_0(A)-n_0(B) - N_{+-},   \label{eq:posin2}    \\
 n_0(A,B) &\geq n_0(A) - n + r, \label{eq:zeroin2}\\
 n_-(A,B) &\leq 3 n - r -n_0(A) -n_0(B)  - N_{++},\label{eq:negin2}   \\
 n_\infty(A,B) &\geq n_0(B) - n + r. \label{eq:infin2}
  \end{align}
In particular, if $r<n$ then \eqref{eq:posin2} and \eqref{eq:negin2} are always more informative, because for a singular pencil it holds
$$ n- r < 2 \min(n_0(A),n_0(B)) = n_0(A)+n_0(B) - |\delta|.$$
(Proof of the last claim: $2 \max (\rank(A),\rank(B)) \leq 2 r < n+r,$ where the last inequality holds as we assume $r<n$.)
\end{remark}

We conclude by pointing out a potentially practically important consequence of Theorem~\ref{thm:main}: if 
$B$ is nearly positive definite, then the inertia of $A-z B$ is not very different from that of $A$. This is explored further in the next section.

\section{Applications and extensions}\label{sec:applications}


\subsection{Bounding the number of real eigenvalues in an interval}\label{sec:interval}
The classical Sylvester's law of inertia is useful 
for counting the number of real eigenvalues on an interval, say, $(a,b)$ with $a<b$. 
For example, when $B \succ 0$, 
$|n_+(A-a B)-n_+(A-b B)|$ 
gives the exact number of the eigenvalues of  $(A,B)$ that lie in the interval 
$(a,b)$. We denote this number by $n_{(a,b)}(A,B)$. A similar notation is used for other types of subsets of the real line, e.g., $n_{[a,b]}(A,B)$ for the number of eigenvalues in the closed interval $[a,b]$.
We also denote by $n_a(A,B)$ and $n_b(A,B)$ the number of eigenvalues of $(A,B)$ that are equal to $a$ and $b$ respectively.

Here we investigate a possible extension of such results to a Hermitian indefinite pair $(A,B)$. Mathematically, this is a simple matter of applying a M\"obius transformation : given $\alpha,\beta,\gamma,\delta \in \R$ satisfying $\alpha \delta  \neq \beta \gamma$, $\delta a = \beta$, $\gamma b = \alpha$ (it is always possible to find four such numbers provided that $a\neq b$), consider the map
\begin{equation}\label{eq:Mobius}
A - B z \mapsto C - y D, \qquad C:= \delta A - \beta B, \quad D:=  \alpha B - \gamma A .  \end{equation}
so that $(a,b)$ is mapped to $(0,\infty)$. 
$n_+(C,D)$ is precisely equal to the number of eigenvalues of the pencil $A-z B$ in the interval $(a,b)$. 
We thus obtain 
\begin{corollary}\label{cor:mobius}
Let $a,b\in\mathbb{R}$ and $A,B$ be Hermitian matrices, and denote the inertia 
of the matrix $A-aB$ by 
$\inertiAaB$, and  that of $bB-A$ by $\inertibBA$. Also define 
\[ n_{++} = n_+(A-aB) + n_+(bB-A), \quad 
n_{--}=n_-(A-aB)+n_-(bB-A), \]  
\[ n_{+-} = n_+(A-aB) + n_-(bB-A),\quad 
n_{-+}= n_-(A-aB)+n_+(bB-A),   \]
\[ \delta = n_0(A-aB) - n_0(bB-A), \]
and $N_{++}=\max(n_{++},n_{--}), N_{+-}=\max(n_{+-},n_{-+})$.
Then 
  \begin{align*}
N_{++} -n &\leq n_{(a,b)}(A,B) \leq 2 n - |\delta| - N_{+-},    \\
\delta& \leq  n_a(A,B)\leq 3n - N_{++} - N_{+-} - n_0(bB-A),  \\
N_{+-} - n &\leq n_{(-\infty,a)\cup (b,\infty)\cup \infty}(A,B) \leq 2 n - |\delta| - N_{++},\\
0 &\leq  n_C(A,B)\leq 2 \min(n_+(A-aB),n_-(A-aB),n_+(bB-A),n_-(bB-A)).     \\
-\delta &\leq  n_b(A,B)\leq 3n - N_{++}- N_{+-}- n_0(A-aB),\\
|s(B)|& \leq  n_R(A,B)\leq n - n_0(bB-A). 
  \end{align*}
\ignore{
  \begin{align*}
\max(n_{++},n_{--}) -n &\leq n_{(a,b)}(A,B) \leq 2 n - |\delta| - \max(n_{+-},n_{-+}),    \\
\delta& \leq  n_a(A,B)\leq 3n - \max(n_{++},n_{--}) - \max(n_{+-},n_{-+})  - n_0(bB-A),  \\
\max(n_{+-},n_{-+}) - n &\leq n_{(-\infty,a)\cup (b,\infty)}(A,B) \leq 2 n - |\delta| - \max(n_{++},n_{--}),\\
0 &\leq  n_C(A,B)\leq 2 \min(n_+(A-aB),n_-(A-aB),n_+(bB-A),n_-(bB-A)).     \\
-\delta &\leq  n_\infty(A,B)\leq 3n - \max(n_{++},n_{--}) - \max(n_{+-},n_{-+}) - n_0(A-aB),\\
|s(B)|& \leq  n_R(A,B)\leq n - n_0(bB-A). 
  \end{align*}
}
In particular, we have the simple lower bound
\begin{equation}  \label{eq:simple}
\max(|n_+(A-aB) - n_+(A-bB)|,|n_+(A-aB) - n_+(A-bB)|) \leq n_{[a,b]}(A,B) .
\end{equation}
\end{corollary}
The proof is essentially a substitution $A\mapsto A-aB, B\mapsto bB-A$ in that of Theorem~\ref{thm:attainable}, and the bound~\eqref{eq:simple} is a straightforward consequence of the intermediate value theorem. Note the interval is closed here, which is necessary. 

Note that in this case the congruence transformation (in the sense discussed before Remark~\ref{canonical}) applies to the matrices $C$ and $D$, which are linear combinations of $A,B$. 

\subsection{Inertia of ``nearly definite" matrix pencils}
A typical case where we believe Theorem~\ref{thm:main} would be useful
is when $B$ is ``almost'' definite, in that 
either most of its eigenvalues are positive or most are negative. 
Then the inertia of $A$ gives a good estimate for that of $(A,B)$:
\begin{corollary}\label{cor:almostposdef}
Let $A,B$ be $n\times n$ Hermitian matrices, and suppose that $B$ has 
$n-k$ positive eigenvalues. 
Assume that $k$ is small enough so that $n-k\geq |n_-(A)-n_+(A)|$ holds. Then 
\begin{equation}
  \begin{split}  \label{eq:cor}
n_+(A) -k &\leq n_+(A,B) \leq n_+(A) +k,\\
n_-(A) -k &\leq n_-(A,B) \leq n_-(A) +k,\\
n-2k& \leq n_R(A,B)\leq n.
  \end{split}
  \end{equation}
\end{corollary}
\begin{proof}
The bounds follow from~\eqref{eq:posin},~\eqref{eq:negin}
 by direct calculations, 
taking 
$n_{++}\geq n_{--}$ and $n_{+-}\leq n_{-+}$,  
 which hold by assumption. 
\end{proof}
 
 Very similar results hold when 
 either $A$ or $B$ is ``almost'' definite. We omit the details as the statements and proofs are essentially the same -- recall that 
$A-zB$, $zB-A$, $B-zA$, and $zA-B$ all have the same number of positive eigenvalues and the same number of negative eigenvalues.

\subsubsection{Parity of algebraic multiplicities}
Assuming moreover that $a,b$ are not eigenvalues of $A-\lambda B$, then for the algebraic multiplicities,
we can further identify the even/odd parity that $n_{(a,b)}(A,B)$ can take, as follows. For concreteness we treat the case where $B$ has at least as many positive eigenvalues as negative ones; analogous results hold in the opposite case, and they can be derived by taking $B\mapsto -B$.
\begin{theorem}\label{thm:almostdef}
Let $A,B$ be $n\times n$ Hermitian matrices, and suppose that $B$ has 
$n-k$ positive eigenvalues, with $n\geq 2k$, and that neither $a$ nor $b$ is an eigenvalue of the pencil $A-zB$. 
Then, 
\begin{equation}
  \label{eq:abeignum}
 n_{(a,b)}(A,B) = |n_+(A - a B)-n_+(A-b B)|    
   + 2h
\end{equation}
for some  $h \in \{0,1,2,\dots, k \}$. 
In particular, 
  \begin{equation}
    \label{eq:ABrealnum}
|n_+(B)-n_-(B)|    \leq n_R(A,B)     \leq n.
  \end{equation}

\end{theorem}
\begin{proof}
It is clear that 
 \eqref{eq:ABrealnum} can be obtained from~\eqref{eq:abeignum} taking $a=-b$ and letting $b \rightarrow + \infty$.
It remains to prove~\eqref{eq:abeignum}. 

Denote  $a_+=n_+(A-aB)$ and $b_+=n_+(A-bB)$, and suppose that the rank of $A-zB$ is $r$. Then, $n-r$ eigenfunctions are identically zero.
 Recall~\cite[Sec. 5.2]{MNTX16} that for a nonzero continuous function $f(x)$ such that $f(a)f(b) \neq 0$, the \emph{local type of $f(x)$ in the interval $[a,b]$} is the ordered pair $(\sign f(a),\sign f(b))$. Denote by $\ell_\tau$ the number of eigenfunctions of local type $\tau$ on $[a,b]$. Then we have the 
  constraints 
$$ \ell_{(+,+)} + \ell_{(+,-)} = a_+, \qquad \ell_{(+,+)}+\ell_{(-,+)}=b_+,$$
yielding
$$\ell_{(-,+)} - \ell_{(+,-)} =  b_+ - a_+.$$
Noting that a root of even multiplicity corresponds to an even number of eigenvalues (counting algebraic multiplicities),~\cite[Prop.~5.5]{MNTX16} implies that an eigenfunction corresponds to an odd number of eigenvalues (counting algebraic multiplicities) in $(a,b)$ if and only if its  local type is either $(+,-)$ or $(-,+)$. Hence, 
 $n_{(a,b)}(A,B) - |a_+-b_+|$ must be even, i.e., equal to $2h$ for some nonnegative integer $h$.

Suppose $h>k$ and let $t_0>0$ be large enough that the inertia of $A+t_0 B$ and $A-t_0B$ are, respectively, $(n-k,n-r,r+k-n)$ and $(r+k-n,n-r,n-k)$. 
Using~\eqref{eq:simple} with $a\mapsto -t_0,b\mapsto a$ and $a\mapsto b,b\mapsto t_0$, and noting that $a,b$ are not eigenvalues, we see that the number of real eigenvalues is bounded below by
$$ |n-k-a_+| + |a_+-b_+| + |b_+ - k| + 2 h \geq n - 2k + 2h > n,$$
a contradiction. Finally, to show $h \geq 0$, we can invoke the intermediate value theorem as before.

\end{proof}

In particular, when $k$ is known to be small, ~\eqref{eq:abeignum} shows that $|n_+(A-a B)-n_+(A-b B)|$ provides a good lower bound of $n_{(a,b)}(A,B)$. 
A summary of this subsection is that the inertia of Hermitian matrices $A-t B$  for $t\in\mathbb{R}$ gives some information about the real eigenvalues of $(A,B)$, but not sufficient to count the exact number of eigenvalues on an interval.

%

\section{Nonlinear eigenvalue problems}\label{sec:nonlinear}
We now turn to nonlinear eigenvalue problems. 
In this setting, we study $F(z)$, an analytic\footnote{In fact, our results are based on the intermediate value theorem, which holds more generally for continuous matrix-valued functions. We assume analyticity because (1) working over continuous functions requires several extra technicalities (2) in practice, functions of interest are analytic at least in appropriate regions of the complex plane.} matrix-valued function of $z$, which takes Hermitian values for $z \in \R$. Again, its finite eigenvalues are defined as the values of $\lambda \in \C $ such that
\begin{equation}\label{eq:defeig2}
\rank F(\lambda)<\max_{z \in \C } \rank F(z) =: r.
\end{equation}
Examples include polynomial eigenvalue problems 
$F(z) = \sum_{i=0}^k z^iA_i$, $A_i\in\mathbb{C}^{n\times n}$. We assume $A_i=A_i^*$ are all symmetric (or Hermitian); such $F$ are called symmetric (Hermitian) matrix polynomials. As in the linear case, we formally define $F(\infty):=A_k$ and seek eigenvalues in $\C^*$.
As before, we explore Sylvester-like inertia laws for $F$, or rather, lower bounds for the number of eigenvalues in an interval $(a,b)$. 
Obtaining upper bounds appears to be not possible, at least by the approach presented here: indeed, the total number of eigenvalues of $F(z)$ cannot be bounded in general, and in fact it might be not even countable if $F(z)$ is not analytic. There are exceptions, e.g., when $F(z)$ is polynomial (in which case the number of eigenvalues is at most its rank times its degree~\cite{dopicoindexsum}). It is sometimes possible to 
quantify the number of eigenvalues in a given domain when $F(z)$ is assumed analytic~\cite{bindel2015localization} or, again, polynomial~\cite{BNS13, NST15}. 

 
Along the lines of the previous sections, we use the following straightforward result. 

\begin{proposition}\label{prop:matpoly}
Let $F(z)$ be a Hermitian matrix function, with elements that depend analytically in $z$. Let $a<b \in \R$ and suppose that neither $a$ nor $b$ is an eigenvalue of $F(z)$. Then,
the nonlinear eigenvalue problem $F(z)$ has at least
$|n_+(F(a))-n_+(F(b))|$ real eigenvalues in the open interval
 $(a,b)$. 
\end{proposition}
\begin{proof}
Recalling Theorem~\ref{thm:conteigs}, we examine the $n$ eigenfunctions 
$\lambda_i(t)$, which are continuous: $n_+(F(a))$ of them must be positive at $t=a$ and 
$n_+(F(b))$ of them are positive at $t=b$. Hence, at least $|n_+(F(a))-n_+(F(b))|$ of them must cross the axis of the abscissae at least once. 
\end{proof}

Proposition~\ref{prop:matpoly}, while elementary to prove, can be used to show
easily that the class of \emph{definite matrix polynomials} introduced in~\cite{higham2009definite} have only real eigenvalues, and also to determine their types, which are defined so that they correspond to the signs of the lowest order nonzero derivatives of $\lambda_i(t)$ at the zeros.  
Such matrix polynomials $(F(z):=) P(z)=\sum_{i=0}^k z^iA_i$ of degree $k$ are 
those for which there exist $\{\mu_i\}_{i=1}^{k+1}$ such that $(-1)^{i}P(\mu_i)$ is positive (or negative) definite for all $i$. Then, applying Theorem~\ref{prop:matpoly} with $a\mapsto \mu_i,b\mapsto \mu_{i+1}$ shows $P$ has (at least) $n$ eigenvalues in $(\mu_i,\mu_{i+1})$. Since this holds for all $i$, it follows that all the $nk$ eigenvalues are real; and hence each interval $(\mu_i,\mu_{i+1})$ contains exactly $n$ eigenvalues. We can further derive the sign characteristics of the eigenvalues as the sign of the derivatives of $\lambda_i(t)$ at the zeros, which are all the same for the $n$ eigenvalues in each interval (and the signs must differ between neighboring intervals). 

Proposition~\ref{prop:matpoly} has little resemblance to the original Sylvester's law of inertia, as there is no congruence transformations involved. A weak result can be obtained along this line when $P$ is a matrix polynomial $P(z) = \sum_{i=0}^kz^iA_i$. Specifically, take $a=0,b=\infty$. Then $|n_+(P(a))-n_+(P(b))|$ is determined solely by the leading and trailing coefficients $A_0$ and $A_k$, since 
$n_+(P(a))=n_+(A_0)$ and $n_+(P(b))=n_+(A_k)$. Since $n_+(A_0),n_+(A_k)$ in turn are invariant under congruence transformations $A_0\mapsto X^*A_0X,$ $A_k\mapsto Y^*A_kY$, it follows that any matrix polynomial of the form
 $\widetilde P(z) = z^kY^*A_kY+X^*A_0X+\sum_{i=1}^{k-1}z^iA_i$
 has at least $|n_+(A_0)-n_+(A_k)|$ real positive eigenvalues where $A_i$ for $i=1,\ldots,k-1$ are arbitrary Hermitian matrices and $X,Y$ are arbitrary invertible matrices. Similarly, $\widetilde P(z)$ has at least 
$|n_+(A_0)-n_+(A_k)|$ (resp. $|n_+(A_0)-n_-(A_k)|$) real negative eigenvalues if $k$ is even (resp. odd).

\section{Geometric multiplicities}\label{sec:geom}
Our results so far have been concerned with the algebraic multiplicities of eigenvalues. We now comment on the geometric counterparts. For a definition of algebraic and geometric multiplicities for analytic matrix functions, see~\cite[Def.~2.3]{MNTX16}.
It transpires that most bounds are valid for the geometric multiplicities: 
most of the proofs made no reference to the fact that the algebraic multiplicity is counted. 
Therefore, bounds in the following are valid also for geometric multiplicities: 
Theorem~\ref{thm:main}, the bounds \eqref{eq:posin2}--\eqref{eq:infin2}, Corollaries~\ref{cor:mobius} and~\ref{cor:almostposdef}, and Proposition~\ref{prop:matpoly}. 

An exception is Theorem~\ref{thm:almostdef}. There,  the argument in connection with the multiplicities of the roots of the eigenfunctions apply specifically to the algebraic multiplicities. 
For the geometric multiplicity, a ``less informative'' bound holds:
\begin{equation}  \label{eq:abeignumgeom}
|n_+(A - a B)-n_+(A-b B)| \leq  n_{(a,b)}(A,B) 
\leq |n_+(A - a B)-n_+(A-b B)|   + 2k. 
\end{equation}
The lower bound is essentially~\eqref{eq:simple}, and the upper bound is a direct consequence of~\eqref{eq:abeignum} and the fact that the algebraic multiplicity is an upper bound for the geometric multiplicity. 

It is worth noting that multiple eigenvalues are a nongeneric phenomenon, which do not arise e.g. for random symmetric matrices $A,B$. In such cases the two multiplicities are clearly the same, and the formula \eqref{eq:abeignum} is valid for both types of multiplicities. 

\section{Experiments}\label{sec:exp}
\subsection{Number of real eigenvalues for $B$ nearly definite}
We first examine generalized symmetric indefinite eigenvalue problems $Ax=\lambda Bx$
where $B$ is nearly definite. 

In Figure~\ref{fig:xgxs4} we plot the Rellich's eigenfunctions $\lambda_i(t)$ for 
$F(t)=A-Bt$ with $A=A^T, B=B^T \in \R^{7 \times 7}$ and $B$ indefinite. 
 In both figures, the inertia of $B$ is $(6,0,1)$. 
We verify the bounds in \eqref{eq:realin} hold, which tell us  that $(A,B)$ has 
at least $5$ and at most $7$ real eigenvalues.

\begin{figure}[htpb]
   \begin{tabular}{lr}

  \begin{minipage}[t]{0.5\hsize}
      \includegraphics[height=50mm]{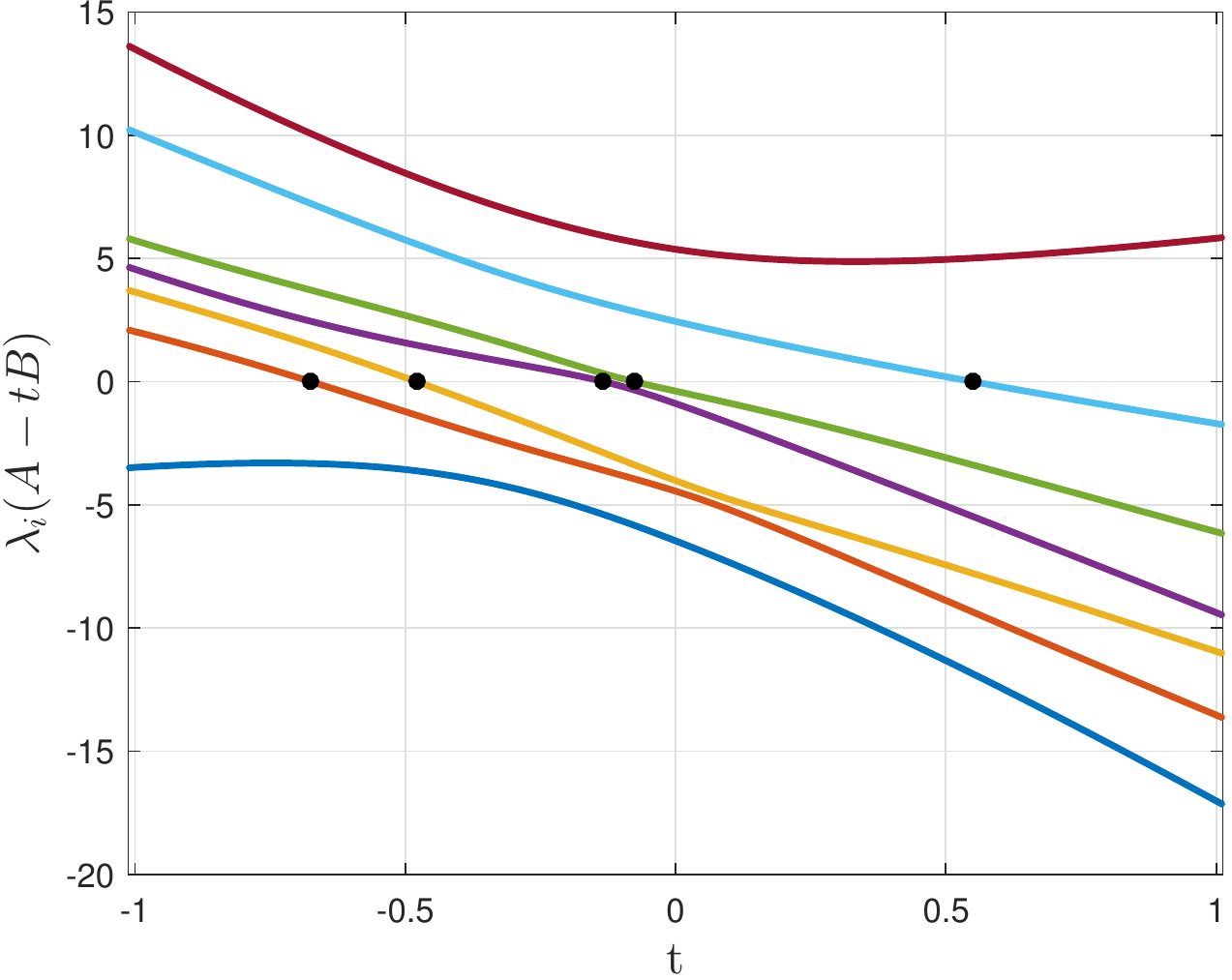}
  \end{minipage} 
  &
  \begin{minipage}[t]{0.5\hsize}
      \includegraphics[height=50mm]{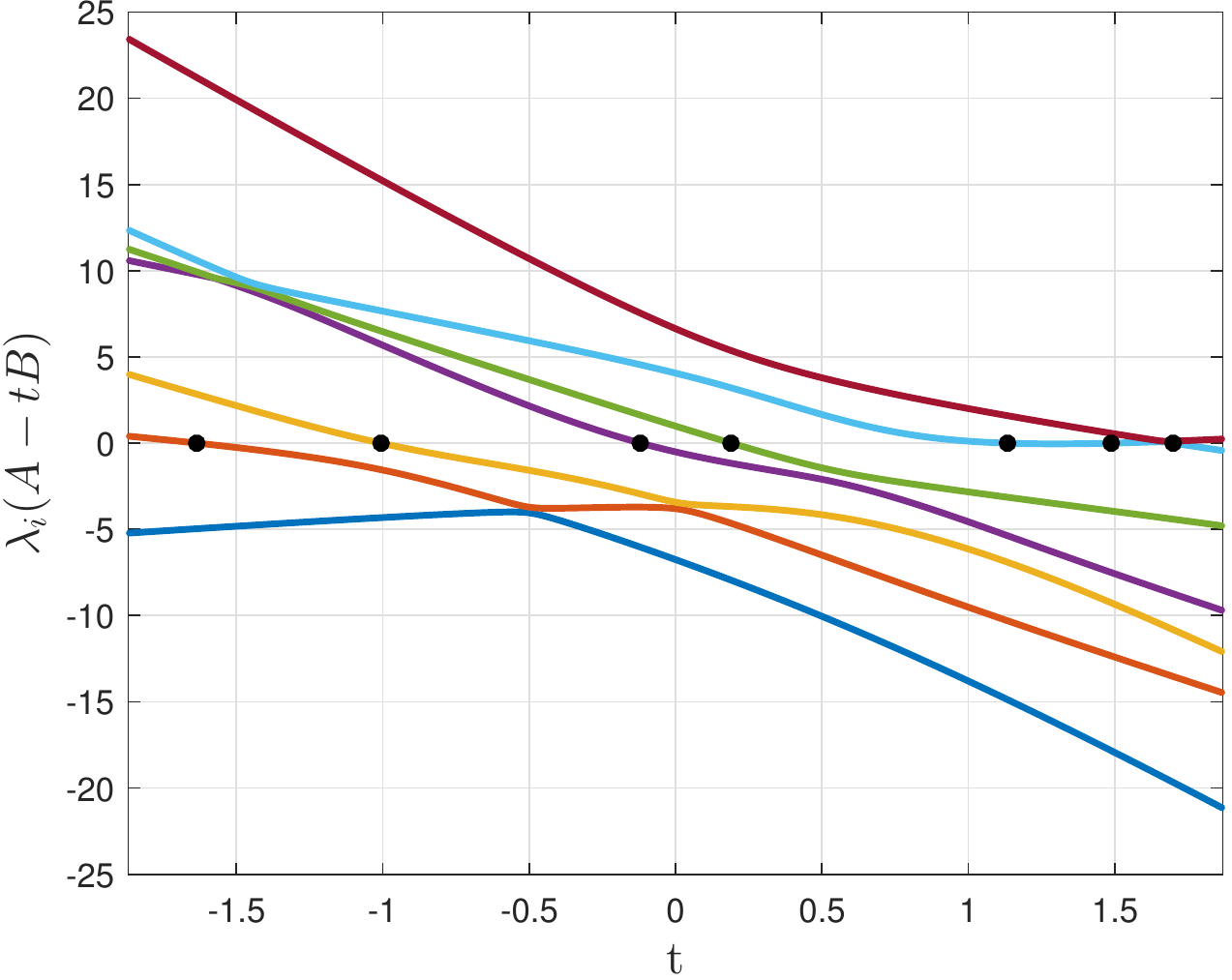}
  \end{minipage}
 \end{tabular}
  \caption{
$\lambda_i(t)$ for two pencils
$A-tB$ with inertia of $B$ fixed to $\inertiB = (6,0,1)$. 
The left pencil has five real eigenvalues, while the right has $n=7$. 
}
  \label{fig:xgxs4}
\end{figure}

Note also the evidently observable phenomenon of eigenvalue repulsion effect in Figures~\ref{fig:xgxs4}, which has been the hallmark of eigenvalue theory in the field of symmetric parameterized eigenvalue problems~\cite[Sec.~9.5]{Laxlaa}. We emphasize the fact that this does not hold if the parameterized matrices are nonsymmetric. 

\subsection{Variation of the inertia of  $(A,B)$ when the inertia of $A,B$ is fixed}
Next we set $B$ as before, but now fix also the inertia $A$ to $(5,0,2)$.
By Theorem~\ref{thm:syl}, all such $A$'s are congruent to each other, and so are all such $B$'s. 
As before, such $(A,B)$ must have at least five real eigenvalues, but more can be said: Theorem~\ref{thm:main} shows that $(A,B)$ has at least four positive eigenvalues, and one negative eigenvalue. Figure~\ref{fig:jj} illustrates this and the sharpness of the bounds. 
\begin{figure}[htpb]
   \begin{tabular}{lr}
  \begin{minipage}[t]{0.5\hsize}
      \includegraphics[height=50mm]{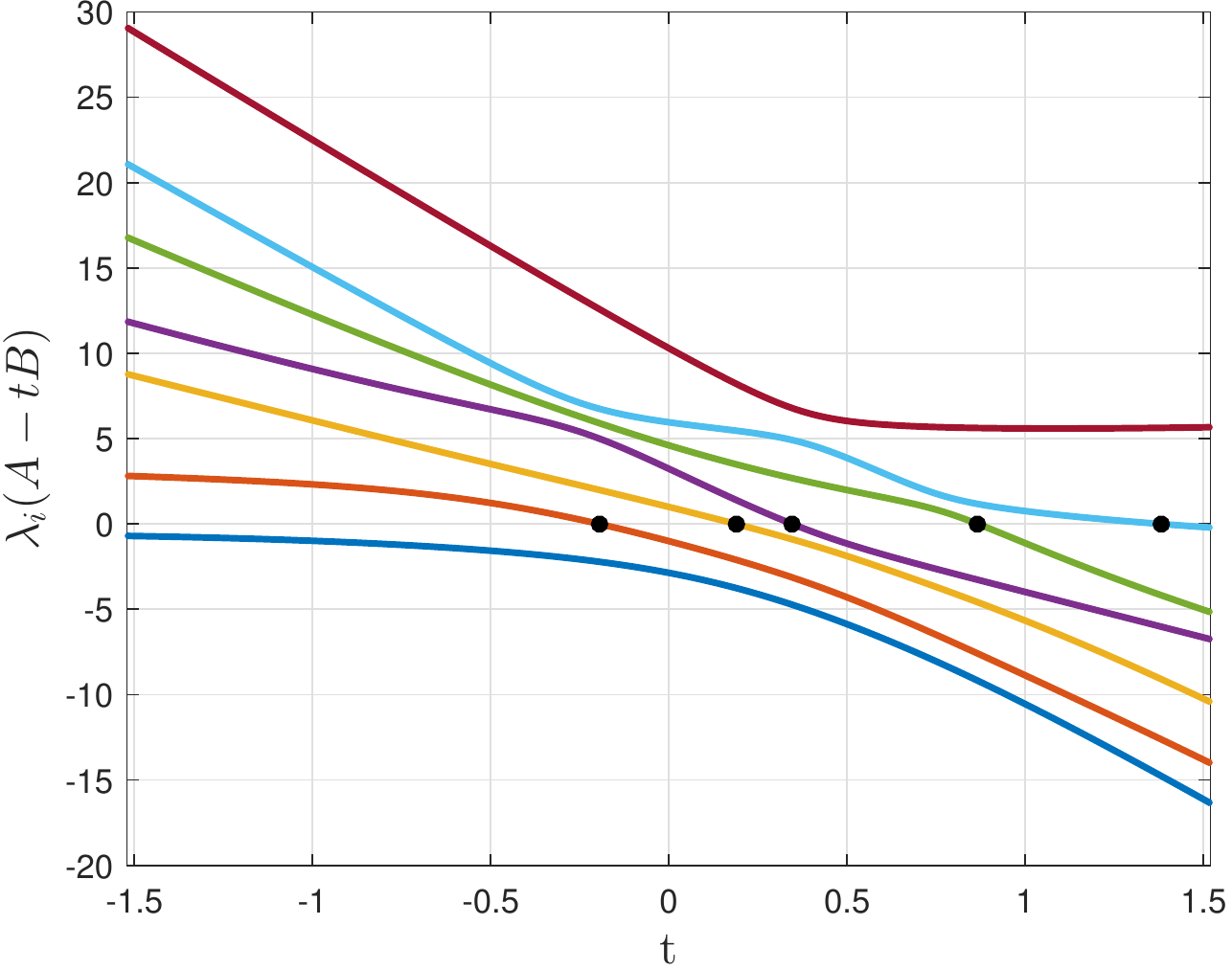}
  \end{minipage} 
  &
  \begin{minipage}[t]{0.5\hsize}
      \includegraphics[height=50mm]{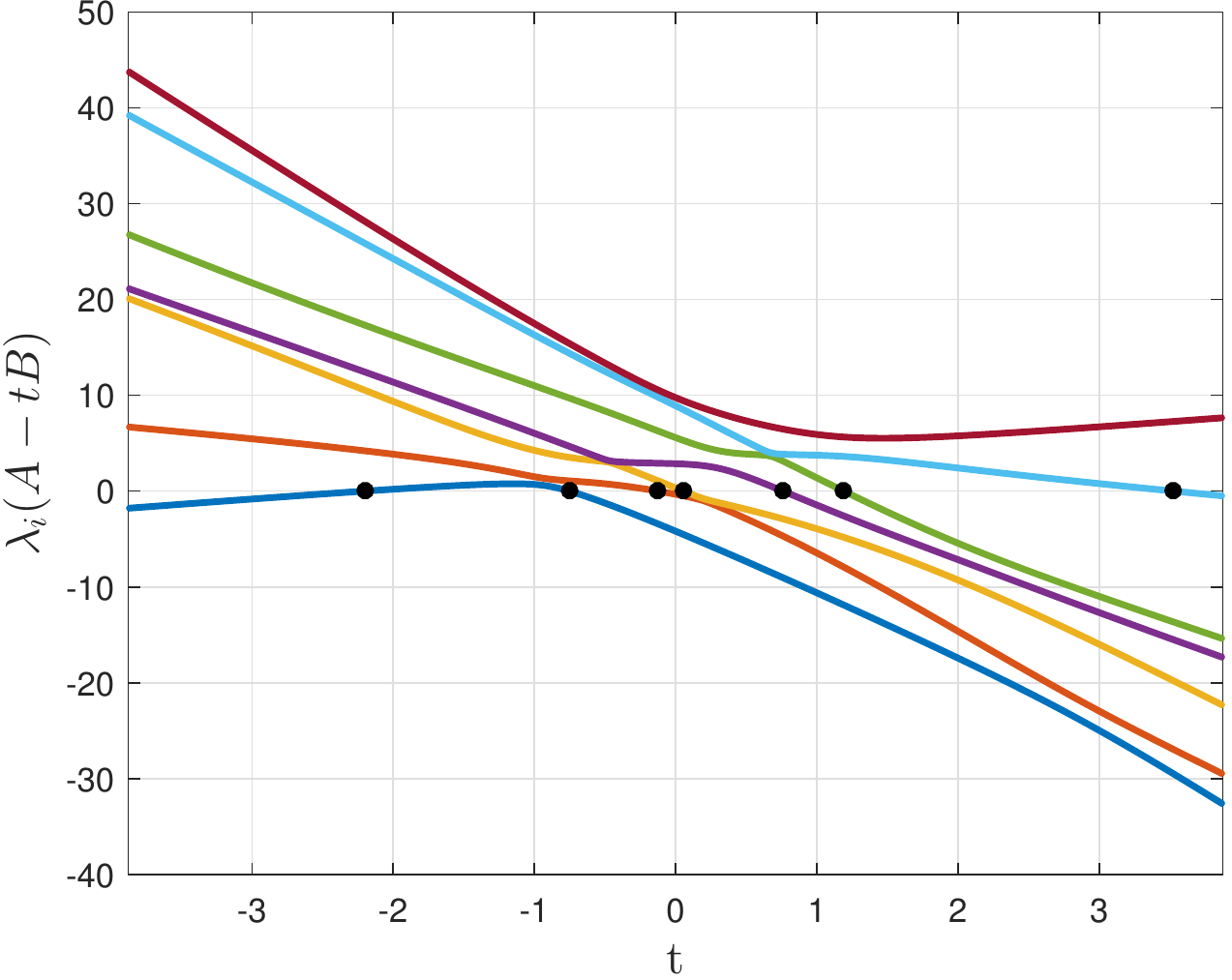}
  \end{minipage}
 \end{tabular}
  \caption{
Same setting as in Figure~\ref{fig:xgxs4}, but with 
 inertia of $A$ fixed to $(n_+(A),n_0(A),n_-(A)) = (5,0,2)$. 
The left pencil has four positive and one negative eigenvalues, while all the eigenvalues of the right pencil are real. 
} 
  \label{fig:jj}
\end{figure}

Using the bound~\eqref{eq:abeignum}, 
 we can obtain significant information on the number of eigenvalues lying in intervals. For example, by computing the inertia of $A-tB$ at $t=-\frac{1}{2}$, which is $(6,0,1)$,  together with that of $B$, we see that both pencils in Figure~\ref{fig:jj} satisfy $n_{(t,\infty)}=4+\{0,2\}$ and  $n_{(-\infty,t)}=\{0,2\}$ (the algebraic multiplicity), and the right pencil 
has $n_{(t,\infty)}=4+\{0,2\}$ and  $n_{(-\infty,t)}=\{0,2\}$. The simple bound~\eqref{eq:simple} is sharp in the left pencil, and $h=1$ in~\eqref{eq:abeignum} for $n_{(t,\infty)}$ of the right pencil.


\subsection{Estimating the number of eigenvalues in an interval}\label{ex:estimate}
Equation~\eqref{eq:abeignum} in Theorem~\ref{thm:almostdef} gives the possible number of eigenvalues of $A-\lambda B$ that lie in an interval $(a,b)$ from $n_+(A - a B)$, $n_+(A-b B)$ and $n_+(B)$ (we make the generic assumption that the eigenvalues are simple). 
Here we illustrate the result and examine a typical value of $2h$ in~\eqref{eq:abeignum}. We generated random symmetric matrices $A,B$ by the MATLAB\footnote{MATLAB is a registered trademark by The MathWorks, Inc.} code {\tt X = randn(n); X = X+X';}, and then shift $B\mapsto B-sI$ , where $s=(\lambda_{k}(B)+\lambda_{k+1}(B))/2$, so that $B$ has exactly $k$ negative eigenvalues. 
We then compute the bound \eqref{eq:abeignum}, along with the exact number 
$s$ of eigenvalues in $(a,b)$, and compute $(2h:=)s-|n_+(A - a B)-n_+(A-b B)| $, the underestimation count of the lower bound $|n_+(A - a B)-n_+(A-b B)|$. 
We took $n=5000,(a,b) = (0,1)$ and generated $1000$ such random instances, and Figure~\ref{fig:hist} shows their histograms for $k=10$ (left) and $k=100$ (right). 
We observe that a typical value of $h$ lies around $0.3k$ in this configuration, and the tail of the distribution becomes lighter as $k$ grows. 
As long as $n$ is large enough, the figures are largely independent of the matrix size $n$; for example, $n = 1000$ gave very similar results.

\begin{figure}[htbp]
  \begin{minipage}[t]{0.5\hsize}
  \centering
      \includegraphics[height=50mm]{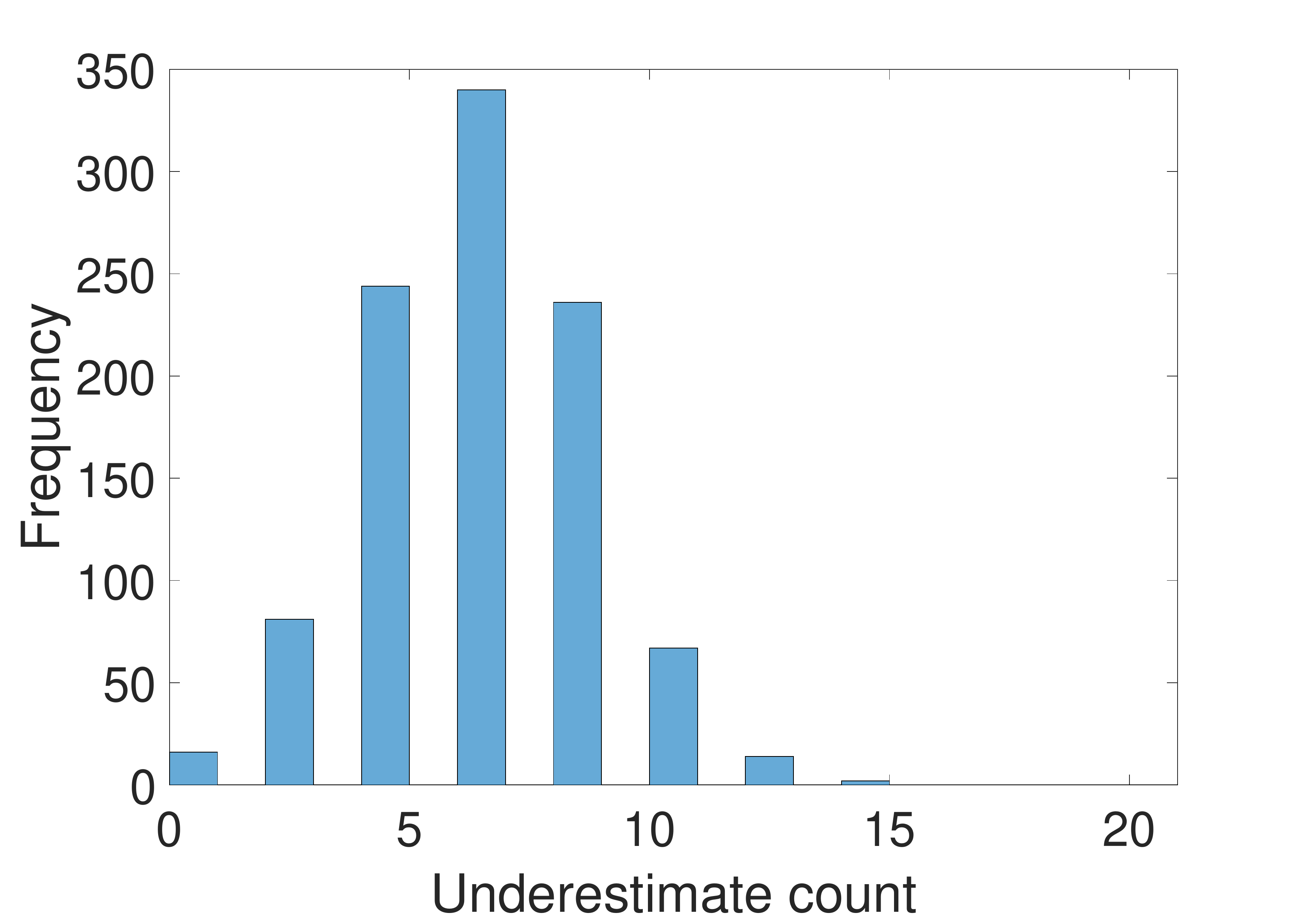}  
  \end{minipage}   
  \begin{minipage}[t]{0.5\hsize}
  \centering
      \includegraphics[height=50mm]{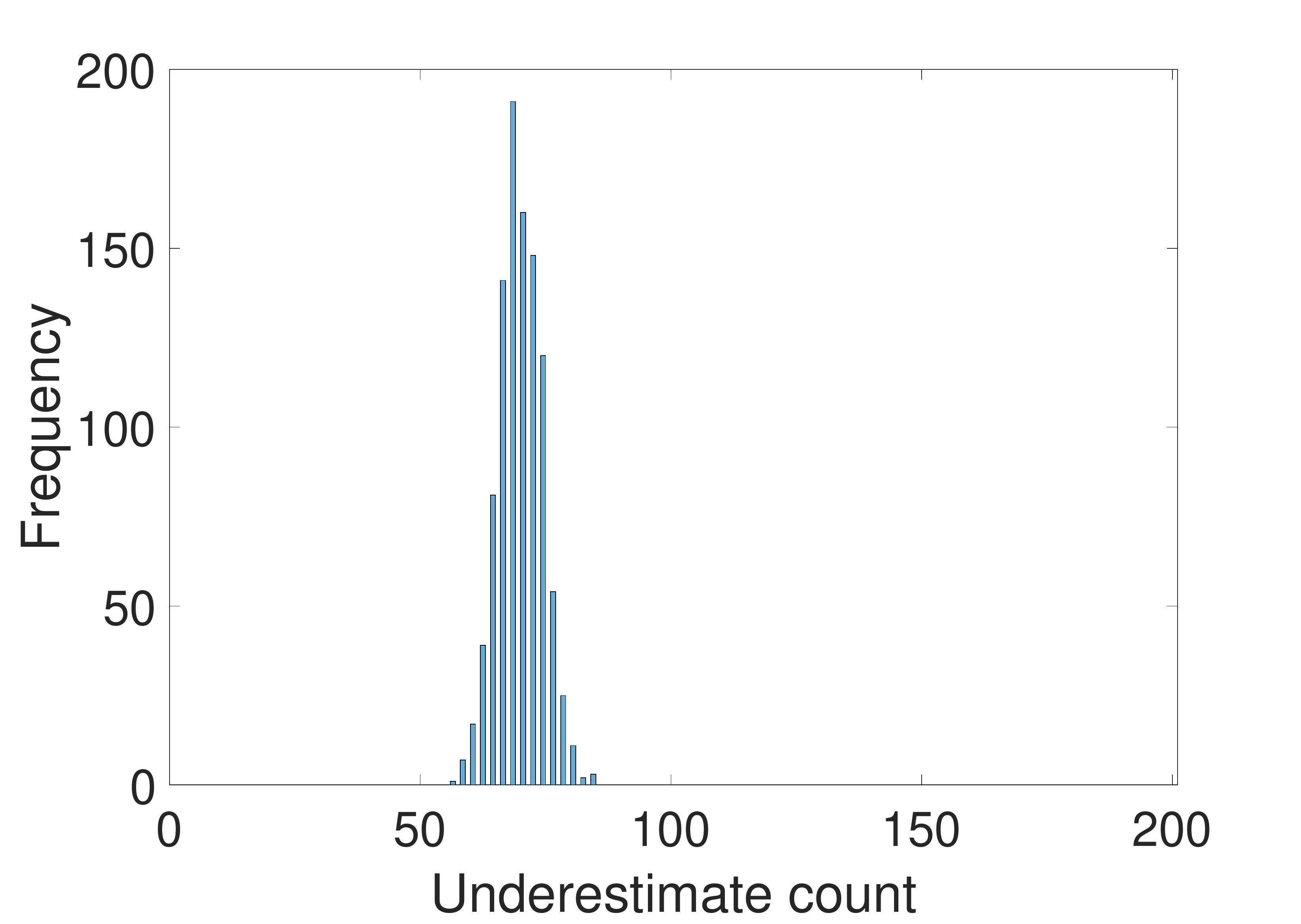}  
  \end{minipage}
  \caption{Histograms from 1000 randomly generated instances of $2h\in\{0,2,4,\ldots,2k\}$ (left: $k=10$, right: $k=100$)  in Theorem~\ref{thm:almostdef}, the difference between the exact number of eigenvalues and its lower bound $|n_+(A - a B)-n_+(A-b B)| $. $A,B\in\mathbb{R}^{5000\times 5000}$. 
}
  \label{fig:hist}
\end{figure}

\subsection{Quadratic matrix polynomials}
\paragraph{Hyperbolic case}
We now turn to hyperbolic quadratic matrix polynomials 
$P(\lambda) = \lambda^2A+\lambda B+C$ with 
$A,B,C\in\mathbb{R}^{n\times n}$ symmetric and $0\succ A, C\succ 0$, which have only real eigenvalues (a special case of definite matrix polynomials). 
As discussed 
in Section~\ref{sec:nonlinear}, it is easy to see why the eigenvalues are all real: the $n$ Rellich's eigenfunctions $\lambda_i(t)$ (setting $F(t)=A t^2 + B t + C$ in Theorem~\ref{thm:conteigs}) alternate in sign between $-M,0,$ and $M$ for any large enough positive constant $M$. 
Figure~\ref{fig:quad} plots the eigenfunctions for an example with  $n=5$.

In fact, for any interval $(a,b)$, one can compute the exact number 
of eigenvalues in it from the inertia of $P(a)$ and $P(b)$: 
for example when $a<0<b$, $n_-(P(b))+n_-(P(a))$ is the number of eigenvalues.
In particular, we can bound the largest (or smallest) eigenvalue $\lambda_{\max}(P)$ (or $\lambda_{\min}(P)$) of $P$ as follows: examine the inertia of $P(t)=t^2A+tB+C$ for $t>0$ ($t<0$), and if they are all negative, then $\lambda_{\max}(P)<t$ ($\lambda_{\min}(P)>t$). 

\begin{figure}[htbp]
  \centering
      \includegraphics[height=50mm]{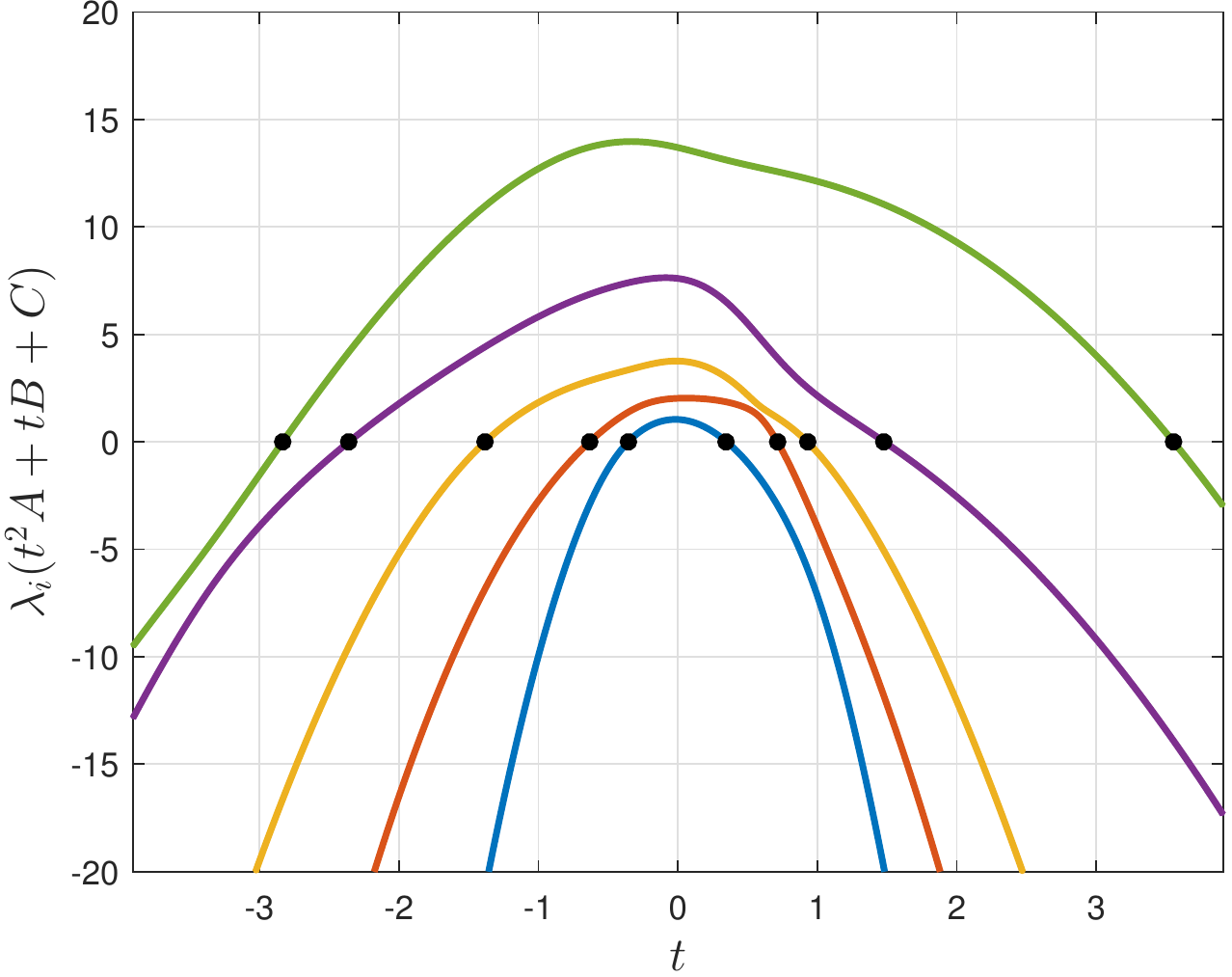}  
  \caption{
Rellich's eigenfunctions for a $5\times 5$ quadratic hyperbolic 
$P(t) = A t^2+ Bt +C$. Note that $\lambda_i(t)$ are all negative for sufficiently small $t$, all positive at $t=0$, and all negative again at sufficiently large $t$. 
}
  \label{fig:quad}
\end{figure}


\paragraph{NLEVP's {\tt spring} example}
We consider the quadratic polynomial eigenvalue problem 
$P(\lambda) = \lambda^2A+\lambda  B+C$ 
in NLEVP~\cite{betcke2013nlevp}  {\tt nlevp('spring',n,1,10*ones(n,1))}, with matrices 
\[
A = I, \quad B =
\beta 
\begin{bmatrix}
20& -10 &&&& \\  
-10& 30 &\ddots&&& \\  
& \ddots &\ddots&\ddots&& \\  
&  &\ddots&30&-10 \\  
&  &&-10&20 \\  
\end{bmatrix}
 \quad C =
\begin{bmatrix}
15& -10 &&&& \\  
-10& 15 &\ddots&&& \\  
& \ddots &\ddots&\ddots&& \\  
&  &\ddots&15&-10 \\  
&  &&-10&15 \\  
\end{bmatrix}. 
\]
As explained in~\cite{ght09a}, for values of $\beta$ larger than $\approx 0.52$, $P$ is definite, but it is indefinite for smaller values of $\beta$. 
For example, taking $\beta=0.3$ and $n=7$, we plot the 
eigenvalues of $F(t) = A t^2+ Bt +C$ in Figure~\ref{fig:quadspring}. 
The inertiae of the matrix $F(t)$ at $t=-13$ and $t=-4$ are 
$(7,0,0)$ and $(3,0,4)$, respectively. 
Thus Proposition~\ref{prop:matpoly} guarantees that $P$ has at least $4$ eigenvalues in the interval $(-13,-4)$. This bound happens to be sharp here; the effectiveness of this approach evidently depends on the choice of the interval (for example, taking the inertia at $t=-15$ and $t=0$ we obtain no useful information). 
\begin{figure}[htbp]
  \begin{minipage}[t]{0.5\hsize}
  \centering
      \includegraphics[height=50mm]{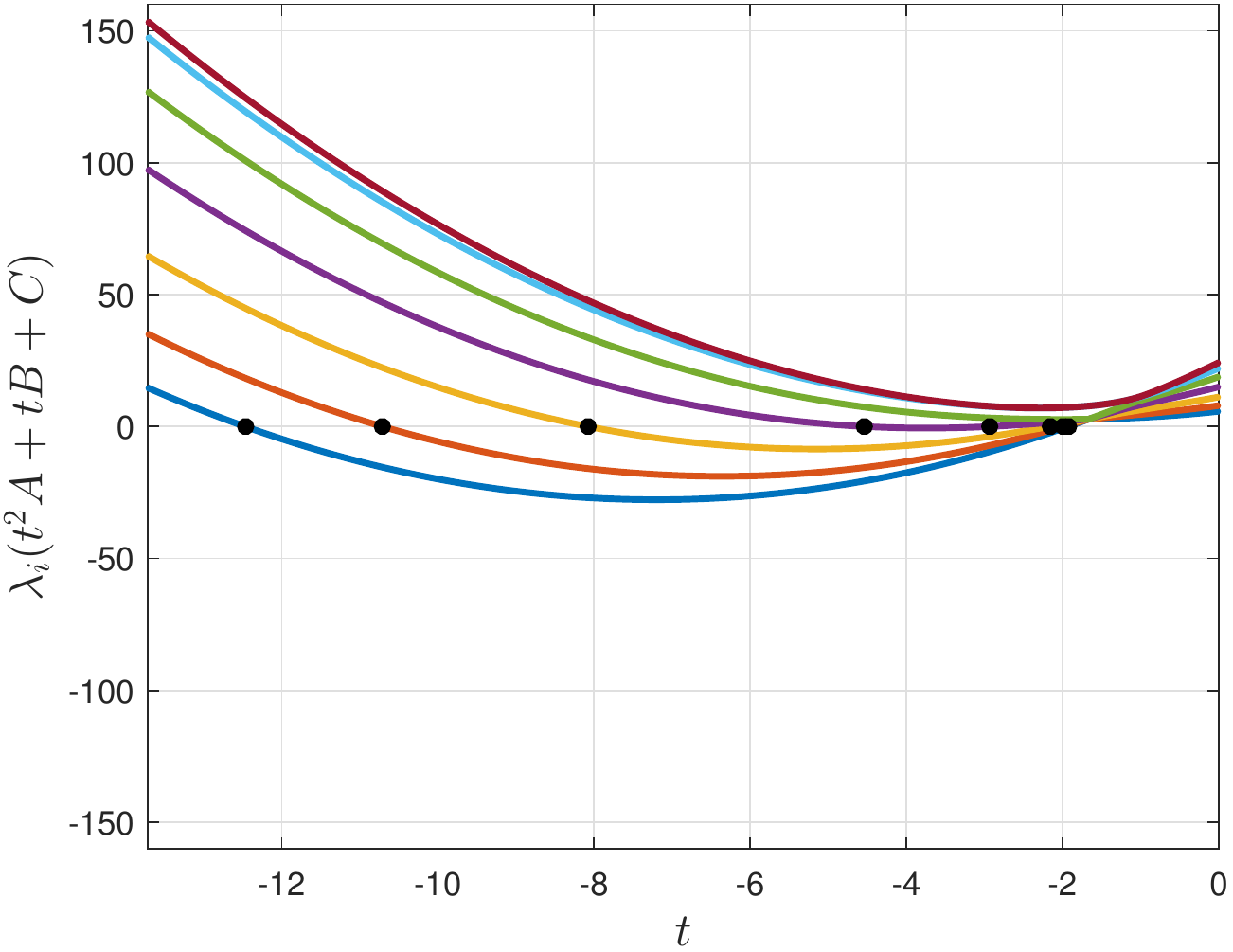}  
  \end{minipage}   
  \begin{minipage}[t]{0.5\hsize}
  \centering
      \includegraphics[height=50mm]{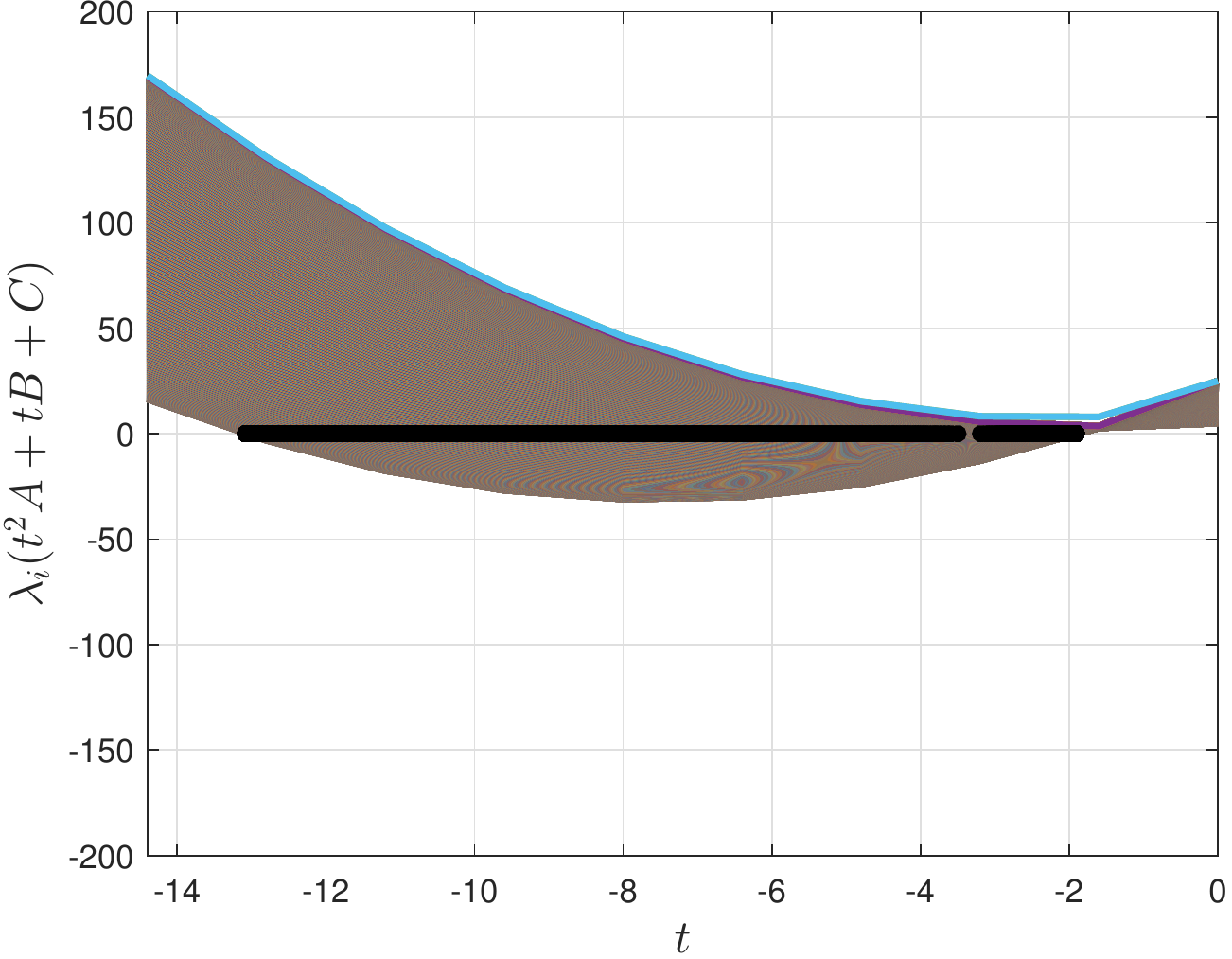}  
  \end{minipage}
  \caption{
Rellich's eigenfunctions for NLEVP's {\tt spring} problem.
$F(t) = A t^2+ \beta Bt +C$, $\beta=0.3$, $n=7$ (left) and $n=1000$ (right). 
This is an indefinite example, but by Proposition~\ref{prop:matpoly} we can obtain nontrivial information on the number of eigenvalues in an interval. 
}
  \label{fig:quadspring}
\end{figure}

The same approach can be used in larger scale. For example, taking $n=1000$ in the same example, and examining the inertia of $F(t)$ at $t = -14$ and $-3.3$, we see that there are at least $626$ eigenvalues in the interval $(-14,-3.3)$, which also happens to be the correct value (see Figure~\ref{fig:quadspring} (right)). 
With the current fastest algorithms avaliable for each task, computing the inertia of $F(t)$ at two values of $t$
(which can be done via computing two $LDL^T$ factorization of $n\times n$ symmetric matrices) is more efficient than computing the eigenvalues of $P$ (for which a standard method requires the solution of a $2n\times 2n$ eigenvalue problem). One can also obtain estimates of the inertia of Hermitian matrices more efficiently by estimating the spectral density~\cite{specdensity2016}.

\subsection{
When Jordan blocks are present}
In all experiments so far, the eigenvalues are simple, and hence the algebraic and geometric multiplicities are always the same. 
Here we examine the case when they are different, that is, when Jordan blocks of size two or larger are present. 
A Jordan block in the canonical form of Hermitian pairs is 
\begin{equation}  \label{eq:jordan}
A =
\begin{bmatrix}
&&&1&\lambda  \\
&&1&\lambda  &\\
&\adots&\adots  &&\\
1&\lambda  &&\\
\lambda  &&
\end{bmatrix}, \qquad 
B =
\begin{bmatrix}
&&&&1  \\
&&&1  &\\
&&\adots  &&\\
&1  &&\\
1  &&
\end{bmatrix}. 
\end{equation}
In such cases, 
the eigenfunctions have zeros of high multiplicity, and Theorem~\ref{thm:main} is not tight.
 We illustrate this in Figure~\ref{fig:jordan}, where $A,B$ are the $6\times 6$ case of~\eqref{eq:jordan}. 
Here the plots of the eigenfunctions of $A-t B$ reveal the eigenvalue $1$ and its geometric multiplicity 1, but it is impossible from the figure alone to identify the algebraic multiplicity $6$ (for this goal, we would need to study the plots of the first, second, $\dots$, fifth derivatives of the $\lambda_i(t)$ as well). 
Theorem~\ref{thm:main}  gives only trivial bounds
$0\leq n_+(A,B)\leq n$, $0\leq n_-(A,B)\leq n$.
The inertia of $A-tB$ is $(3,0,3)$ for any value of $t\neq 1$, and 
our results, e.g.,~\eqref{eq:abeignumgeom} provide no information on the geometric multiplicities of the eigenvalues from the inertia of $A-tB$, unless we take $t=1$. 
Theorem~\ref{thm:almostdef} shows that the algebraic multiplicities of real eigenvalues (if any) are even. Recall that~\eqref{eq:abeignum} in Theorem~\ref{thm:almostdef} is not satisfied for the geometric multiplicity.

\begin{figure}[htbp]
  \centering
      \includegraphics[height=50mm]{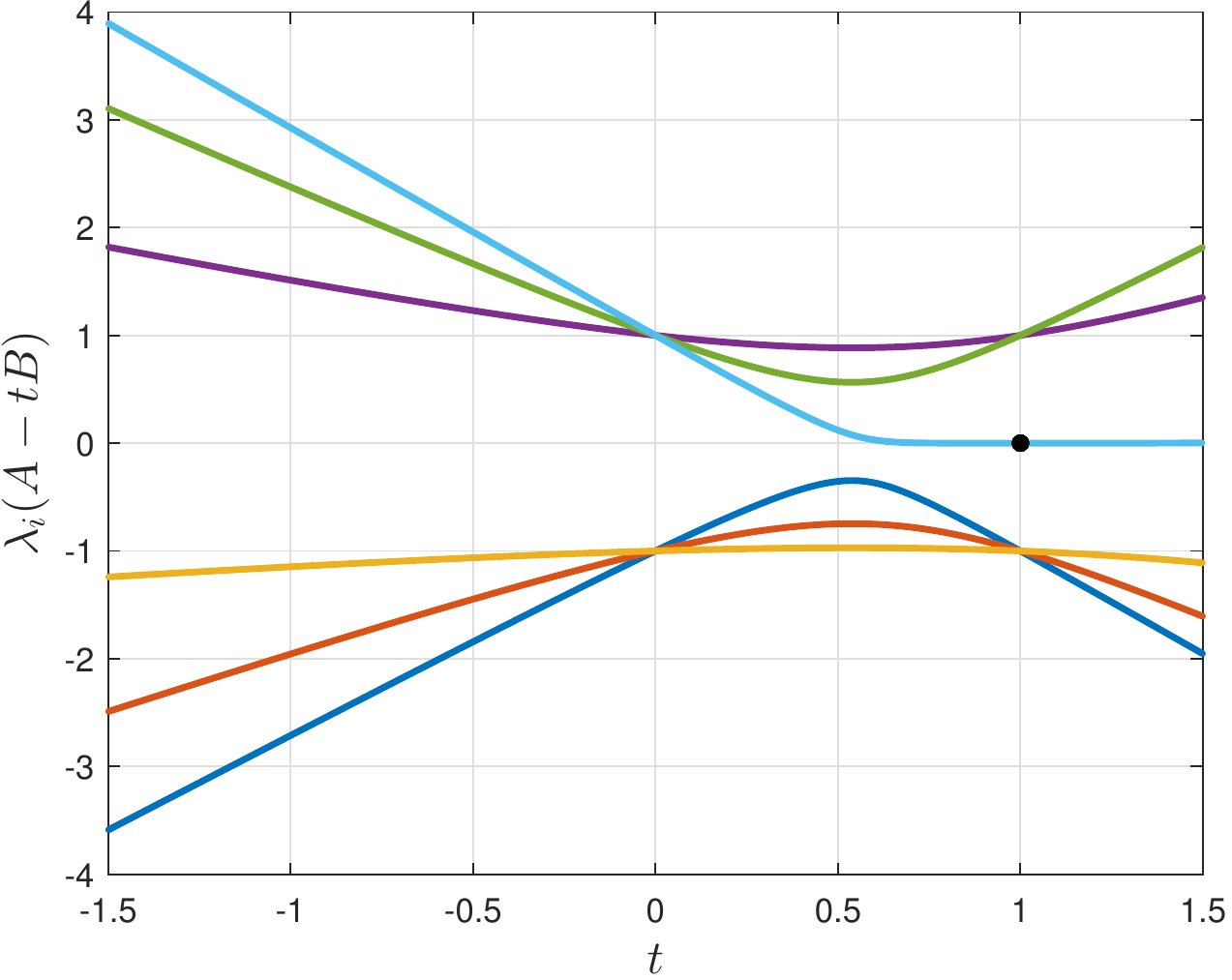}  
  \caption{
Rellich's eigenfunctions of the $A-tB$ for $A,B$ as in~\eqref{eq:jordan}, 
a $6\times 6$ pair with Jordan block of size $6$ at $\lambda=1$. 
}
  \label{fig:jordan}
\end{figure}

We note that, unlike the previous plots,
the curves intersect at $t=0$ and $t=1$ in this example. This nongeneric behavior happens because $A,B$ are in such special forms such that the pencil $A-zB$ has a Jordan block.

\subsection*{Acknowledgement}
We thank Andy Wathen for pointing out the role played by eigenvalues of indefinite pencils in constraint preconditioners.


\begin{thebibliography}{10}

\bibitem{adachi2017solving}
S.~Adachi, S.~Iwata, Y.~Nakatsukasa, and A.~Takeda.
\newblock Solving the trust-region subproblem by a generalized eigenvalue
  problem.
\newblock {\em SIAM J. Optim.}, 27(1):269--291, 2017.

\bibitem{betcke2013nlevp}
T.~Betcke, N.~J. Higham, V.~Mehrmann, C.~Schr{\"o}der, and F.~Tisseur.
\newblock Nlevp: A collection of nonlinear eigenvalue problems.
\newblock {\em ACM Trans. Math. Soft.}, 39(2):7, 2013.

\bibitem{bilirchicone98}
B.~Bilir and C.~Chicone.
\newblock A generalization of the inertia theorem for quadratic matrix
  polynomials.
\newblock {\em Linear Algebra Appl.}, 280:229--240, 1998.

\bibitem{bindel2015localization}
D.~Bindel and A.~Hood.
\newblock Localization theorems for nonlinear eigenvalue problems.
\newblock {\em SIAM Rev.}, 57(4):585--607, 2015.

\bibitem{BNS13}
D.~A. Bini, V.~Noferini, and M.~Sharify.
\newblock Locating the eigenvalues of matrix polynomials.
\newblock {\em SIAM J. Matrix Anal. Appl.}, 34(4):1708--1727, 2013.

\bibitem{dopicosymlin}
M.~I. Bueno, F.~M. Dopico, S.~Furtado, and M.~Rychnovsky.
\newblock Large vector spaces of block-symmetric strong linearizations of
  matrix polynomials.
\newblock {\em Linear Algebra Appl.}, 477:165--210, 2015.

\bibitem{Lathauwer2006}
L.~De~Lathauwer.
\newblock A link between the canonical decomposition in multilinear algebra and
  simultaneous matrix diagonalization.
\newblock {\em SIAM J. Matrix Anal. Appl.}, 28(3):642--666 (electronic), 2006.

\bibitem{dopicoindexsum}
F.~De~Ter\'{a}n, F.~M. Dopico, and D.~S. Mackey.
\newblock Spectral equivalence of matrix polynomials and the index sum theorem.
\newblock {\em Linear Algebra Appl.}, 459:264--333, 2014.

\bibitem{demmelbook}
J.~Demmel.
\newblock {\em Applied Numerical Linear Algebra}.
\newblock SIAM, Philadelphia, USA, 1997.

\bibitem{di2016efficient}
E.~Di~Napoli, E.~Polizzi, and Y.~Saad.
\newblock Efficient estimation of eigenvalue counts in an interval.
\newblock {\em Numer. Lin. Alg. Appl.}, 23(4):674--692, 2016.

\bibitem{ght09a}
C.-H. Guo, N.~J. Higham, and F.~Tisseur.
\newblock An improved arc algorithm for detecting definite {Hermitian} pairs.
\newblock {\em SIAM J. Matrix Anal. Appl.}, 31(3):1131--1151, 2009.

\bibitem{higham2009definite}
N.~J. Higham, D.~S. Mackey, and F.~Tisseur.
\newblock Definite matrix polynomials and their linearization by definite
  pencils.
\newblock {\em SIAM J. Matrix Anal. Appl.}, 31(2):478--502, 2009.

\bibitem{hornjohn}
R.~A. Horn and C.~R. Johnson.
\newblock {\em Matrix {A}nalysis}.
\newblock Cambridge University Press, second edition, 2012.

\bibitem{ikramov01}
K.~D. Ikramov.
\newblock On the inertia law for normal matrices.
\newblock {\em Dokl. Math.}, 64:141£ü142, 2001.

\bibitem{kato}
T.~Kato.
\newblock {\em Perturbation Theory for Linear Operators}.
\newblock Springer-Verlag, 2nd edition, 1966.

\bibitem{keller2000constraint}
C.~Keller, N.~I.~M. Gould, and A.~J. Wathen.
\newblock Constraint preconditioning for indefinite linear systems.
\newblock {\em SIAM J. Matrix Anal. Appl.}, 21(4):1300--1317, 2000.

\bibitem{kosti2013sylvester}
A.~Kosti\'{c} and H.~Voss.
\newblock On {S}ylvester's law of inertia for nonlinear eigenvalue problems.
\newblock {\em Electron. Trans. Numer. Anal}, 40:82--93, 2013.

\bibitem{lancaster2005canonical}
P.~Lancaster and L.~Rodman.
\newblock Canonical forms for {H}ermitian matrix pairs under strict equivalence
  and congruence.
\newblock {\em SIAM Rev.}, 47(3):407--443, 2005.

\bibitem{Laxlaa}
P.~Lax.
\newblock {\em Linear Algebra and Its Applications}.
\newblock Wiley, 2nd edition, 2007.

\bibitem{specdensity2016}
L.~Lin, Y.~Saad, and C.~Yang.
\newblock Approximating spectral densities of large matrices.
\newblock {\em SIAM Rev.}, 58(1):34--65, 2016.

\bibitem{Mackey05vectorspaces}
D.~S. Mackey, N.~Mackey, C.~Mehl, and V.~Mehrmann.
\newblock Vector spaces of linearizations for matrix polynomials.
\newblock {\em SIAM J. Matrix Anal. Appl.}, 28:971--1004, 2005.

\bibitem{MNTX16}
V.~Mehrmann, V.~Noferini, F.~Tisseur, and H.~Xu.
\newblock On the sign characteristics of {H}ermitian matrix polynomials.
\newblock {\em Linear Algebra Appl.}, 511:328--364, 2016.

\bibitem{biroots}
Y.~Nakatsukasa, V.~Noferini, and A.~Townsend.
\newblock Computing the common zeros of two bivariate functions via
  {B}{\'e}zout resultants.
\newblock {\em Numer. Math.}, 129:181--209, 2015.

\bibitem{m4revsimax}
Y.~Nakatsukasa, V.~Noferini, and A.~Townsend.
\newblock Vector spaces of linearizations for matrix polynomials: a bivariate
  polynomial approach.
\newblock {\em SIAM J. Matrix Anal. Appl.}, 38(1):1--29, 2017.

\bibitem{NST15}
V.~Noferini, M.~Sharify, and F.~Tisseur.
\newblock Tropical roots as approximations to eigenvalues of matrix
  polynomials.
\newblock {\em SIAM J. Matrix Anal. Appl.}, 36(1):138--157, 2015.

\bibitem{rellich37}
F.~Rellich.
\newblock St\"{o}rungstheorie der {S}pektralzerlegung {I}.
\newblock {\em Math. Anal.}, 113:600--619, 1937.

\bibitem{rellichbook}
F.~Rellich.
\newblock {\em Perturbation Theory of Eigenvalue Problems}.
\newblock Gordon and {B}reach, 1969.

\bibitem{syl1852}
J.~J. Sylvester.
\newblock A demonstration of the theorem that every homogeneous quadratic
  polynomial is reducible by real orthogonal substitutions to the form of a sum
  of positive and negative squares.
\newblock {\em Philosoph. Mag.}, 4(32):138£ü142, 1852.

\end{thebibliography}
\end{document}